\documentclass[amssymb,twoside,11pt]{article}
\usepackage{latexsym,amsmath,graphicx}
\usepackage[dvipsnames]{xcolor}
\usepackage{placeins}
\usepackage{amsfonts}
\usepackage{enumitem}
\newtheorem{theorem}{Theorem}[section]
\newtheorem{Lemma}[theorem]{{\bf Lemma}}

\newtheorem{rem}[theorem]{{\bf Remark}}
\newtheorem{ex}[theorem]{{\bf Example}}
\newtheorem{definition}{Definition}[section]
\numberwithin{equation}{section}
\newenvironment{proof}{\indent{\em Proof:}}{\quad \hfill
$\Box$\vspace*{2ex}}

\newcommand{\ra}{\rightarrow}

\font\Bbb=msbm10 at 12pt

\newcommand{\R}{\mbox{\Bbb R}}
\newcommand{\N}{\mbox{\Bbb N}}

\begin{document}
\setcounter{page}{1}
\begin{center}
\vspace{0.4cm} {\large{\bf Nonlocal Boundary Value Problem for Generalized Hilfer Implicit Fractional Differential Equations }} \\

\vspace{0.35cm}
Ashwini D. Mali $^{1}$\\
maliashwini144@gmail.com\\

\vspace{0.5cm}
Kishor D. Kucche $^{2}$ \\
kdkucche@gmail.com \\

\vspace{0.35cm}
$^{1,2}$ Department of Mathematics, Shivaji University, Kolhapur-416 004, Maharashtra, India.

\end{center}

\def\baselinestretch{1.0}\small\normalsize

\begin{abstract}
In this paper, we derive the equivalent fractional integral equation to the nonlinear implicit fractional differential equations involving $\varphi$-Hilfer fractional derivative subject to nonlocal fractional integral boundary conditions.
The existence of a solution, Ulam-Hyers, and Ulam-Hyers-Rassias stability has been acquired by means equivalent fractional integral equation. Our investigations depend on the fixed point theorem due to Krasnoselskii and the Gronwall inequality involving  $\varphi$-Riemann--Liouville fractional integral. An example is provided to show the utilization of primary outcomes.

\end{abstract}

\def\baselinestretch{1.5}
\allowdisplaybreaks
\section{Introduction}
The fundamental study about initial value problem (IVP) for Caputo implicit fractional differential equations (FDEs) of the form 
\begin{align*} 
{}^{c}D^{\alpha}x(t)= f(t,x(t), {}^{c}D^{\alpha}x(t)), ~x(0)=x_0,
\end{align*} 
have been initiated by Nieto et al. \cite{Nieto}. The major investigation  relating the existence and uniqueness of the solution and  Ulam-Hyers stabilities for implicit FDEs with different kinds of initial and boundary conditions can be found in the work of Benchohra et al.  \cite{Bench1, Bench2, Bench3, Bench4, Bench5, Bench6, Bench7} . Kucche et al. \cite{Kucche1} investigated existence, the interval of existence and uniqueness of solutions along with various qualitative properties of solution for the implicit FDEs mentioned above. 

There have appeared numerous significant works about the nonlinear boundary value problems( BVPs) for FDEs, out of which we are mentioning here only a few that are relating to the works of the present paper. Existence and uniqueness results for nonlocal BVPs of nonlinear FDEs involving Caputo fractional derivative has been examined by Benchohra et al.\cite{Bench8} and Ahmad and Nieto\cite{Ahmad} through various types of fixed point theorems. Utilizing the methodology of \cite{Ahmad}  the study has been extended to nonlocal BVPs for nonlinear integrodifferential equations \cite{Ahmad1}. Zhang\cite{S. Zhang}, obtained results relating to the existence and multiplicity of positive solutions of BVP for nonlinear Caputo FDEs. In \cite{Jiang}, Jiang and Wang investigated results concerning the existence of solutions of multi-point BVP  for Riemann-Liouville FDEs.

On the other hand, Hilfer\cite{Hilfer} defined a two parameter fractional derivative called Hilfer fractional derivative that incorporates both fractional operators, Riemann-Liouville derivative  and Caputo derivative. The fundamental  work on the theory of IVP for FDEs involving Hilfer derivative can be found in \cite{Furati}.  Asawasamrit et al. \cite{Asawasamrit} initiated the study of BVPs for FDEs involving Hilfer fractional derivatives subject to nonlocal integral boundary conditions. The BVP for fractional integrodifferential equations with Hilfer derivative has been researched in \cite{Thabet} for existence and data dependence of solutions. The implicit FDEs with
a nonlocal condition involving Hilfer fractional derivative investigated in \cite{Bench9, Vivek} for existence and Ulam types stabilities.

The Hilfer version of the fractional derivative with another function called $\varphi$-Hilfer fractional derivative has been presented by Jose et al. \cite{Vanterler1}. The basic study about existence and uniqueness of the solution of a nonlinear $\varphi$-Hilfer FDEs with different kinds of initial conditions and the  Ulam-Hyers and Ulam-Hyers-Rassias stabilities of its solutions have been explored  in \cite{jose1, jose2, vanterler3, j1, j3, Mali1, Mali2, Kharade} . The implicit FDEs  involving $\varphi$-Hilfer derivative has been investigated in \cite{jose3} for the existence and uniqueness of the solution and the  Ulam-Hyers-Rassias stability.

Thinking about available literature, it is seen that the study of boundary value problems for nonlinear implicit FDEs involving generalized fractional derivative  is still in the underlying stages and numerous parts of this theory should be investigated. Inspired by the work of the papers mentioned above and the work of \cite{Bench8, Asawasamrit},  the primary goal of the present work is to establish the existence the solutions and to examine the Ulam-Hyers stabilities of the following  implicit FDEs involving $\varphi$-Hilfer fractional derivative subject to nonlocal integral boundary conditions,
\begin{align}
 ^H \mathcal{D}^{\mu,\,\nu\,;\, \varphi}_{a +}y(t)
&= f(t, y(t),^H \mathcal{D}^{\mu,\,\nu\,;\, \varphi}_{a +}y(t)), ~t \in  J', ~ 1<\mu<2, ~0\leq\nu\leq 1 \label{eq1}\\
 y(a)&=0,\label{eq2}\\
 y(b)& =\sum_{i=1}^{m}
\lambda_i\,\mathcal{I}_{a +}
^{\delta_i\,;\, \varphi}y(\tau_i), \qquad  \label{eq3}
\end{align}
where $J'=(a,\,b], $ $^H \mathcal{D}^{\mu,\nu;\, \varphi}_{a+}(\cdot)$ is the $\varphi$-Hilfer fractional derivative of order $\mu$ and type $\nu$, $\mathcal{I}_{a +}^{\delta_i;\, \varphi}(\cdot)$ is $\varphi$-Riemann--Liouville fractional integral of order $\delta_i>0$, $\lambda_i\in \R, \, i=1,2,\cdots,m,$\, $0\leq a\leq\tau_1<\tau_2<\tau_3<\cdots< \tau_m\leq b$ and   $f\in C(J' \times \R  \,, \R) $.

The $\varphi$-Hilfer implicit  FDEs with nonlocal integral boundary conditions considered in the present paper is the more broad class of  BVP that incorporates for different values of $\nu$ and $\varphi$ the class of BVP for implicit FDEs involving fractional derivative operators to be specific Riemann-Liouville, Caputo, Hadamard, Katugampola and other fractional derivative operators referenced in \cite{Vanterler1}. In particular, 
\begin{itemize}
\item for  $\varphi(t)=t$, the outcomes acquired in the present paper incorporates the results of \cite{Asawasamrit} for non-implicit Hilfer FDEs with nonlocal BVP.
\item for  $\nu=1,\,\varphi(t)=t, \delta_i=0, a=0, b=1$ and $\tau_1=\tau_2=\dots=\tau_{m-1}=0$, incorporates the results of \cite{Ahmad} for non-implicit Caputo FDEs with  nonlocal BVP.
\end{itemize}

The paper is structured in five sections as follows: 
In sections 2, we present the definitions and the results that are utilized in the paper. Section 3, provide an equivalent fractional integral equation to the nonlocal implicit BVP problem \eqref{eq1}-\eqref{eq3}. Section 4, deals with existence of solution for  nonlocal  BVP problem \eqref{eq1}-\eqref{eq3}. Ulam-Hyers stability and  Ulam-Hyers-Rassias stability has been examined in section 5. In section 6, we provided an example to justify our results. 
\section{Preliminaries} \label{preliminaries}
Let $\xi=\mu+\nu\left(2-\mu \right) $, $1<\mu<2, ~0\leq\nu\leq 1$.  Then $1< \xi \leq 2$. Let $\varphi\in C^{1}(J,\,\R)$ be an increasing function with $\varphi'(t)\neq 0$, for all $\, t\in J=[a,b]$.
 Consider the space 
\begin{equation*}
C_{2-\xi ;\, \varphi }(J,\,\R) =\left\{ \mathfrak{h}:\left( a,b\right]
\rightarrow \mathbb{R}~\big|~\left( \varphi \left( t\right) -\varphi \left(
a\right) \right) ^{2-\xi }\mathfrak{h}\left( t\right) \in C(J,\,\R)
\right\} ,
\end{equation*}
with the norm
\begin{equation}\label{space1}
\left\Vert \mathfrak{h}\right\Vert _{C_{2-\xi ;\,\varphi }\left(J,\,\R\right)  }=\underset{t\in J 
}{\max }\left\vert \left( \varphi \left( t\right) -\varphi \left( a\right) \right)
^{2-\xi }\mathfrak{h}\left( t\right) \right\vert.
\end{equation}
\begin{definition} [\cite{Kilbas}]
Let $\mu>0 ~(\mu \in \R)$, $\mathfrak{h} \in L_1(J,\,\R)$.  Then, the $\varphi$-Riemann--Liouville fractional integral of a function $\mathfrak{h}$ with respect to $\varphi$ is defined by 
\end{definition}
\begin{equation}\label{P1}
\mathcal{I}_{a+}^{\mu \, ;\,\varphi }\mathfrak{h}\left( t\right) =\frac{1}{\Gamma \left( \mu
\right) }\int_{a}^{t}\varphi'(s)(\varphi(t)-\varphi(s))^{\mu-1}\mathfrak{h}(s) \,ds.
\end{equation}

\begin{definition}[\cite{Vanterler1}]  
Let $n-1<\mu <n \in \N $ and $h \in C^{n}(J,\,\R)$. Then, the   $\varphi$-Hilfer fractional derivative $^{H}\mathcal{D}^{\mu,\,\nu\, ;\,\varphi}_{a+}(\cdot)$  of a function $\mathfrak{h}$ of order $\mu$ and type $0\leq \nu \leq 1$, is defined by
\begin{equation}\label{HIL}
^{H}\mathcal{D}_{a+}^{\mu ,\,\nu \, ;\,\varphi }\mathfrak{h}\left(t\right) =\mathcal{I}_{a+}^{\nu \left(
n-\mu \right) \, ;\,\varphi }\left( \frac{1}{\varphi ^{\prime }\left( t\right) }\frac{d}{dt}\right) ^{n}\mathcal{I}_{a+}^{\left( 1-\nu \right) \left( n-\mu
\right) \, ;\,\varphi }\mathfrak{h}\left( t\right).
\end{equation}
\end{definition}

\begin{Lemma}[\cite{Kilbas,Vanterler1}]\label{lema2} 
Let $\mu, \chi>0$ and $\delta>0$. Then
\begin{enumerate}
\item [(a)] $\mathcal{I}_{a+}^{\mu \, ;\,\varphi }\mathcal{I}_{a+}^{\chi \, ;\,\varphi }\mathfrak{h}(t)=\mathcal{I}_{a+}^{\mu+\chi \, ;\,\varphi }\mathfrak{h}(t)$
\item [(b)]$
\mathcal{I}_{a+}^{\mu \, ;\,\varphi }\left( \varphi \left( t\right) -\varphi \left( a\right)
\right) ^{\delta -1}=\frac{\Gamma \left( \delta \right) }{\Gamma \left(
\mu +\delta \right) }\left( \varphi \left(t\right) -\varphi \left( a\right)
\right) ^{\mu +\delta -1}.
$
\item [(c)]$^{H}\mathcal{D}_{a+}^{\mu ,\,\nu \, ;\,\varphi }\left( \varphi \left(t\right) -\varphi \left( a\right)
\right) ^{\xi -1}=0.
$
\end{enumerate}
\end{Lemma}

\begin{Lemma}[\cite{Vanterler1}]\label{lema1} 
If $\mu >0$ and \, $0\leq \alpha <1,$ then $I_{a+}^{\mu \, ;\varphi }(\cdot)$ is bounded from $C_{\alpha \, ;\varphi }\left[ a,b\right] $ to $C_{\alpha \, ;\varphi }\left[ a,b\right] .$ In addition, if $\alpha \leq \mu $, then $I_{a+}^{\mu \, ;\varphi }(\cdot)$ is bounded from $C_{\alpha \, ;\varphi }\left[ a,b\right] $ to $C\left[ a,b\right] $.
\end{Lemma}

\begin{Lemma}[\cite{Vanterler1}]\label{teo1} 
If $\mathfrak{h}\in C^{n}[a,b]$, $n-1<\mu<n$ and $0\leq \nu \leq 1$, then
\begin{enumerate}[topsep=0pt,itemsep=-1ex,partopsep=1ex,parsep=1ex]
\item $
I_{a+}^{\mu \, ;\varphi }\text{ }^{H}\mathbb{D}_{a+}^{\mu ,\nu \, ;\varphi }\mathfrak{h}\left( t\right) =\mathfrak{h}\left( t\right) -\overset{n}{\underset{k=1}{\sum }}\frac{\left( \varphi \left( t\right) -\varphi \left( a\right) \right) ^{\xi -k}}{\Gamma \left( \xi -k+1\right) }h_{\varphi }^{\left[ n-k\right] }I_{a+}^{\left( 1-\nu \right) \left( n-\mu \right) \, ;\varphi }h\left( a\right)$\\
where, $\mathfrak{h}_{\varphi }^{\left[ n-k\right] }\mathfrak{h}(t)=\left( \frac{1}{\varphi'(t)}\frac{d}{dt}\right)^{n-k}\mathfrak{h}(t)$.
\item $
^{H}\mathbb{D}_{a+}^{\mu ,\nu \, ;\varphi }I_{a+}^{\mu \, ;\varphi }\mathfrak{h}\left( t\right)
=\mathfrak{h}\left( t\right).
$
\end{enumerate}
\end{Lemma}

\begin{theorem}[\cite{Vanterler2}]\label{lema4}
 Let $u,$ $v$ be two integrable functions and $g$ be continuous with domain $\left[ a,b\right] .$ Let $\varphi \in C^{1}\left[ a,b\right] $ an increasing function such that $\varphi ^{\prime }\left( t\right) \neq 0$, $\forall ~t\in \left[ a,b\right]$. Assume that
\begin{enumerate}[topsep=0pt,itemsep=-1ex,partopsep=1ex,parsep=1ex]
\item $u$ and $v$ are nonnegative;
\item $g$ is nonnegative and nondecreasing.

\end{enumerate}

If
\begin{equation*}
u\left( t\right) \leq v\left( t\right) +g\left( t\right) \int_{a}^{t}\varphi
^{\prime }\left( s \right) \left( \varphi \left( t\right) -\varphi \left( s
\right) \right) ^{\mu -1}u\left( s \right) ds,
\end{equation*}
then
\begin{equation}\label{jose}
u\left( t\right) \leq v\left( t\right) +\int_{a}^{t}\overset{\infty }{%
\underset{k=1}{\sum }}\frac{\left[ g\left( t\right) \Gamma \left( \alpha
\right) \right] ^{k}}{\Gamma \left( k\alpha \right) }\varphi ^{\prime }\left(
s \right) \left[ \varphi \left( t\right) -\varphi \left(s \right) \right]
^{k\mu -1}v\left( s \right) ds,~t\in \left[ a,b\right].
\end{equation}
Further, if  $v$ is  a nondecreasing function on $[a,b]$ then
$$u(t)\leq v(t)\,E_{\mu}\left( g(t) \Gamma(\mu)\left(\varphi(t)-\varphi(a)\right)^{\mu}\right),$$
where $E_{\mu}(\cdot)$ is the Mittag-Leffler function of one parameter\cite{Gorenflo}, defined as
$$\mathcal{E}_{\mu}(z)=\sum_{k=0}^{\infty}\frac{z^k}{\Gamma(k \mu +1)}.$$
\end{theorem}

\begin{theorem}[\cite{zhou}, Krasnoselskii] \label{Krasnoselskii}

Let $\mathcal{M}$ be a closed, convex, and nonempty subset of
a Banach space $\mathcal{X}$, and $P$, $Q$ the operators such that
\begin{enumerate}[topsep=0pt,itemsep=-1ex,partopsep=1ex,parsep=1ex]
\item $Px + Qy \in \mathcal{M}$ whenever $x, y \in \mathcal{M}$;
\item $P$ is a contraction mapping;
\item $Q$ is compact and continuous.
\end{enumerate}
Then, there exists ${y}^\ast\in \mathcal{M}$ such that ${y}^\ast=P{y}^\ast+Q{y}^\ast.$
\end{theorem}

\section{Equivalent Fractional Integral Equation}
In this section, we derive equivalent fractional integral equation to  the nonlocal BVP  \eqref{eq1}-\eqref{eq3}.

\begin{theorem}\label{thm3.1}
Let $1<\mu<2, ~0\leq\nu\leq 1$ and $h: J'  \to \R $ be a continuous function. Then the nonlocal  BVP for $\varphi$-Hilfer  FDEs
\begin{align}
^H \mathcal{D}^{\mu,\,\nu\,;\, \varphi}_{a +}y(t)
&=h(t),~~ ~t \in  J'=(a,\,b],  ~\label{eq4}\\
y(a)&=0,\label{eq5}\\
y(b)&=\sum_{i=1}^{m}
\lambda_i\,\mathcal{I}_{a +}
^{\delta_i\,;\, \varphi}y(\tau_i)\label{eq6}
\end{align}
is equivalent to 
$$
y(t)=\frac{( \varphi (t)-\varphi (a))^{\xi -1}}{\Lambda \Gamma(\xi)}\left[\sum_{i=1}^{m}
\lambda_i\,\mathcal{I}_{a +}
^{\mu+\delta_i\,;\, \varphi}h(\tau_i)-\mathcal{I}_{a +}
^{\mu\,;\, \varphi}h(b) \right]+\mathcal{I}_{a +}
^{\mu\,;\, \varphi}h(t), \, t\in J,
$$
where $\lambda_i\in\R (i=1, 2, \cdots, m)$ are the constants such that 
\begin{equation}\label{a1}
\Lambda=\frac{( \varphi (b)-\varphi (a))^{\xi -1}}{\Gamma(\xi)}-\sum_{i=1}^{m}
\frac{\lambda_i}{\Gamma(\xi+\delta_i)}( \varphi (\tau_i)-\varphi (a))^{\xi+\delta_i -1}\neq0.
\end{equation}

\end{theorem}
\begin{proof}
Assume that $y$ is the solution of the nonlocal  BVP for $\varphi$-Hilfer  FDEs \eqref{eq4}-\eqref{eq6}. Operating $\varphi$-fractional integral $\mathcal{I}_{a +}
^{\mu\,;\, \varphi}$ on both sides of equation \eqref{eq4} and  using  Lemma \ref{teo1}, we obtain
$$
y\left( t\right) -\sum_{k=1}^{2}\frac{\left( \varphi \left( t\right) -\varphi \left( a\right) \right) ^{\xi -k}}{\Gamma \left( \xi -k+1\right) }y_{\varphi }^{\left[ 2-k\right] }\mathcal{I}_{a+}^{\left( 1-\nu \right) \left( 2-\mu \right)  \, ;\varphi }y\left( a\right)
=\mathcal{I}_{a+}^{\mu \, ;\varphi }h(t), \, t\in J.
$$
But $\left( 1-\nu \right) \left( 2-\mu \right)=2-\xi$. Therefore,
\begin{align*}
y\left( t\right) 
&=\frac{\left( \varphi \left( t\right) -\varphi \left( a\right) \right) ^{\xi -1}}{\Gamma \left( \xi \right) }\left( \frac{1}{\varphi'(t)}\frac{d}{dt}\right) \mathcal{I}_{a+}^{2 -\xi  \, ;\varphi }y\left( t\right)|_{t=a}\\
&\qquad+\frac{\left( \varphi \left( t\right)-\varphi \left( a\right) \right) ^{\xi -2}}{\Gamma \left( \xi-1 \right) }\mathcal{I}_{a+}^{2 -\xi  \, ;\varphi }y\left( t\right)|_{t=a}+\mathcal{I}_{a+}^{\mu \, ;\varphi }h(t)\\
&=\frac{\left( \varphi \left( t\right) -\varphi \left( a\right) \right) ^{\xi -1}}{\Gamma \left( \xi \right) }\,  ^H\mathcal{D}^{\xi-1,\,\nu\,;\, \varphi}_{a +}y\left( t\right)|_{t=a}\\
&\qquad+\frac{\left( \varphi \left( t\right)-\varphi \left( a\right) \right) ^{\xi -2}}{\Gamma \left( \xi-1 \right) }\mathcal{I}_{a+}^{2 -\xi  \, ;\varphi }y\left( t\right)|_{t=a}+\mathcal{I}_{a+}^{\mu \, ;\varphi }h(t).
\end{align*}
Set 
$$
C_1= ^H\mathcal{D}^{\xi-1,\,\nu\,;\, \varphi}_{a +}y\left( t\right)|_{t=a}\quad\text{and}\quad C_2=\mathcal{I}_{a+}^{2 -\xi  \, ;\varphi }y\left( t\right)|_{t=a},\, t\in J.
$$
Then,
\begin{equation}\label{35}
y\left( t\right) =C_1 \,\frac{\left( \varphi \left( t\right) -\varphi \left( a\right) \right) ^{\xi -1}}{\Gamma \left( \xi \right) }\,  + C_2\,\frac{\left( \varphi \left( t\right)-\varphi \left( a\right) \right) ^{\xi -2}}{\Gamma \left( \xi-1 \right) }+\mathcal{I}_{a+}^{\mu \, ;\varphi }h(t), \, t\in J.
\end{equation}
Since $\lim\limits_{t\rightarrow a}\left( \varphi \left( t\right)-\varphi \left( a\right) \right) ^{\xi -2}=\infty,$ in the view of  boundary condition \eqref{eq5}, we must have $C_2=0.$ In this case \eqref{35} becomes
\begin{equation}\label{3.4}
y\left( t\right) =C_1 \,\frac{\left( \varphi \left( t\right) -\varphi \left( a\right) \right) ^{\xi -1}}{\Gamma \left( \xi \right) }\, +\mathcal{I}_{a+}^{\mu \, ;\varphi }h(t), \, t\in J.
\end{equation}
 Next, to determine the constant $C_1$,  we utilize  the boundary condition \eqref{eq6}. Operating $\mathcal{I}_{a +}
 ^{\delta_i\,;\, \varphi}$ on both sides of equation \eqref{3.4},   we obtain
 \begin{align}\label{37}
\mathcal{I}_{a +}^{\delta_i\,;\,\varphi}y(t)
&=\frac{C_1}{\Gamma \left( \xi+\delta_i \right)}\,\left( \varphi \left( t\right) -\varphi \left( a\right) \right) ^{\xi+\delta_i -1}+\mathcal{I}_{a +}^{\mu +\delta_i\,;\,\varphi}h(t).
 \end{align}
From \eqref{eq6} and \eqref{37} we have 
\begin{align}\label{38}
y(b)&=\sum_{i=1}^{m}\lambda_i\,\mathcal{I}_{a +}^{\delta_i\,;\,\varphi}y(\tau_i)\nonumber\\
&=C_1\sum_{i=1}^{m}\frac{\lambda_i}{\Gamma \left( \xi+\delta_i \right)}\left( \varphi \left( \tau_i\right) -\varphi \left( a\right) \right) ^{\xi+\delta_i -1}+\sum_{i=1}^{m}
\lambda_i\,\mathcal{I}_{a +}^{\mu +\delta_i\,;\,\varphi}h(\tau_i).
\end{align}
But from \eqref{3.4} and \eqref{38},  we have
\begin{align}\label{39}
& C_1 \,\frac{\left( \varphi \left(b\right) -\varphi \left( a\right) \right) ^{\xi -1}}{\Gamma \left( \xi \right) }\, +\mathcal{I}_{a+}^{\mu \, ;\varphi }h(b) \nonumber\\
& \qquad=C_1\sum_{i=1}^{m}\frac{\lambda_i}{\Gamma \left( \xi+\delta_i \right)}\left( \varphi \left( \tau_i\right) -\varphi \left( a\right) \right) ^{\xi+\delta_i -1}+\sum_{i=1}^{m}\lambda_i\,\mathcal{I}_{a +}^{\mu +\delta_i\,;\,\varphi}h(\tau_i),
\end{align}
Since $\lambda_i\in\R (i=1, 2, \cdots, m)$ are the constants such that 
$$
\Lambda=\frac{( \varphi (b)-\varphi (a))^{\xi -1}}{\Gamma(\xi)}-\sum_{i=1}^{m}
\frac{\lambda_i}{\Gamma(\xi+\delta_i)}( \varphi (\tau_i)-\varphi (a))^{\xi+\delta_i -1}\neq0,
$$
the equation \eqref{39} can be written as
$$
C_1=\frac{1}{\Lambda}\left[ \sum_{i=1}^{m}\lambda_i\,\mathcal{I}_{a +}^{\mu +\delta_i\,;\,\varphi}h(\tau_i)-\mathcal{I}_{a+}^{\mu \, ;\varphi }h(b)\right].
$$
Thus, equation \eqref{3.4} takes the form
\begin{equation}\label{3.7}
y(t)=\frac{( \varphi (t)-\varphi (a))^{\xi -1}}{\Lambda \Gamma(\xi)}\left[\sum_{i=1}^{m}
\lambda_i\,\mathcal{I}_{a +}
^{\mu+\delta_i\,;\, \varphi}h(\tau_i)-\mathcal{I}_{a +}
^{\mu\,;\, \varphi}h(b) \right]+\mathcal{I}_{a +}
^{\mu\,;\, \varphi}h(t),\, t\in J.
\end{equation}
Which is the desired equivalent fractional integral equation to the problem \eqref{eq4}-\eqref{eq6}.

Conversely, suppose that $y$ is the solution of the fractional integral equation \eqref{3.7}. Operating fractional derivative $^H \mathcal{D}^{\mu,\,\nu\,;\, \varphi}_{a +}$
 on both sides of equation \eqref{3.7} and using the Lemma \ref{lema2} and Lemma \ref{teo1},  we obtain
 \begin{align}\label{311}
^H \mathcal{D}^{\mu,\,\nu\,;\, \varphi}_{a +}y(t)
&=\frac{1}{\Lambda \Gamma(\xi)}\left[\sum_{i=1}^{m}
\lambda_i\,\mathcal{I}_{a +}
^{\mu+\delta_i\,;\, \varphi}h(\tau_i)-\mathcal{I}_{a +}
^{\mu\,;\, \varphi}h(b) \right]\,^H \mathcal{D}^{\mu,\,\nu\,;\, \varphi}_{a +}( \varphi (t)-\varphi (a))^{\xi -1}\nonumber\\
&\qquad+^H \mathcal{D}^{\mu,\,\nu\,;\, \varphi}_{a +}\mathcal{I}_{a +}^{\mu\,;\, \varphi}h(t)\nonumber\\
&=h(t),\, t\in J.
 \end{align}
 This proves  $y$ satisfies the  equation  \eqref{eq4}. Next, we prove that $y$ given by  \eqref{3.7} verifies the boundary conditions. From \eqref{3.7}, clearly
\begin{equation}\label{312}
  y(a)=0.
\end{equation}
  Now we prove that  $y$ satisfy the boundary condition \eqref{eq6}. For each $ i\,(i=1, 2, \cdots, m)$, from   equation \eqref{3.7}  we have
  \begin{align*}
 \mathcal{I}_{a +}^{\delta_i\,;\,\varphi}y(t)
 &=\frac{1}{\Lambda }\left[\sum_{i=1}^{m}
  \lambda_i\,\mathcal{I}_{a +}
  ^{\mu+\delta_i\,;\, \varphi}h(\tau_i)-\mathcal{I}_{a +}
  ^{\mu\,;\, \varphi}h(b) \right]\frac{\left( \varphi \left( t\right) -\varphi \left( a\right) \right) ^{\xi+\delta_i -1}}{\Gamma \left( \xi+\delta_i \right)}\,+\mathcal{I}_{a +}^{\mu +\delta_i\,;\,\varphi}h(t).
 \end{align*} 
 Therefore,
 \begin{align}\label{3.9}
\sum_{i=1}^{m}\lambda_i\,\mathcal{I}_{a+}^{\delta_i\,;\,\varphi}y(\tau_i)
&=\frac{1}{\Lambda}\left[\sum_{i=1}^{m}\lambda_i\,\mathcal{I}_{a+}^{\mu+\delta_i\,;\,\varphi}h(\tau_i)-\mathcal{I}_{a+}^{\mu\,;\,\varphi}h(b)\right]\sum_{i=1}^{m}\frac{\lambda_i}{\Gamma \left( \xi+\delta_i \right)}\left( \varphi \left( \tau_i\right) -\varphi \left( a\right) \right) ^{\xi+\delta_i-1}\nonumber\\
&\qquad+\sum_{i=1}^{m}\lambda_i\,\mathcal{I}_{a +}^{\mu +\delta_i\,;\,\varphi}h(\tau_i).
\end{align}
But from \eqref{a1}, we have
$$
\sum_{i=1}^{m}
\frac{\lambda_i}{\Gamma(\xi+\delta_i)}( \varphi (\tau_i)-\varphi (a))^{\xi+\delta_i -1}
=\frac{( \varphi (b)-\varphi (a))^{\xi -1}}{\Gamma(\xi)}-\Lambda.
$$
Thus, equation \eqref{3.9} reduces to
\begin{align}\label{3.10}
\sum_{i=1}^{m}\lambda_i\,\mathcal{I}_{a+}^{\delta_i\,;\,\varphi}y(\tau_i)
&=\frac{1}{\Lambda}\left[\sum_{i=1}^{m}\lambda_i\,\mathcal{I}_{a+}^{\mu+\delta_i\,;\,\varphi}h(\tau_i)-\mathcal{I}_{a+}^{\mu\,;\,\varphi}h(b)\right]\left( \frac{( \varphi (b)-\varphi (a))^{\xi -1}}{\Gamma(\xi)}-\Lambda\right) \nonumber\\
&\qquad+\sum_{i=1}^{m}\lambda_i\,\mathcal{I}_{a +}^{\mu +\delta_i\,;\,\varphi}h(\tau_i)\nonumber\\
&=\frac{1}{\Lambda}\left[\sum_{i=1}^{m}\lambda_i\,\mathcal{I}_{a+}^{\mu+\delta_i\,;\,\varphi}h(\tau_i)-\mathcal{I}_{a+}^{\mu\,;\,\varphi}h(b)\right] \frac{( \varphi (b)-\varphi (a))^{\xi -1}}{\Gamma(\xi)}+\mathcal{I}_{a+}^{\mu\,;\,\varphi}h(b).
\end{align} 
Now from \eqref{3.7}, we have 
\begin{equation}\label{3.8}
 y(b)=\frac{( \varphi (b)-\varphi (a))^{\xi -1}}{\Lambda \Gamma(\xi)}\left[\sum_{i=1}^{m}
 \lambda_i\,\mathcal{I}_{a +}
 ^{\mu+\delta_i\,;\, \varphi}h(\tau_i)-\mathcal{I}_{a +}
 ^{\mu\,;\, \varphi}h(b) \right]+\mathcal{I}_{a +}
 ^{\mu\,;\, \varphi}h(b).
\end{equation}
From equations \eqref{3.10} and \eqref{3.8} we obtain 
\begin{equation}\label{3.16}
y(b)=\sum_{i=1}^{m}\lambda_i\,\mathcal{I}_{a +}
^{\delta_i\,;\, \varphi}y(\tau_i).
\end{equation}From  \eqref{311}, \eqref{312} and \eqref{3.16}, it follows that the $y$ defined by  \eqref{3.7} satisfies the problem \eqref{eq4}-\eqref{eq6}.
\end{proof}

\begin{theorem}\label{thm3.2}
Let $f: J'\times \R  \to \R $ be a continuous function such that $f\left( \cdot, y(\cdot), ^H \mathcal{D}^{\mu,\,\nu\,;\, \varphi}_{a +}y(\cdot)\right)\in C_{2-\xi;\,\varphi} (J, \R)$ for each $y \in C_{2-\xi;\,\varphi} (J, \R)$. Then, the nonlocal   BVP for $\varphi$-Hilfer implicit FDEs \eqref{eq1}-\eqref{eq3} is equivalent to the fractional integral equation
\begin{equation}\label{3.12}
y(t)=(\varphi(t)-\varphi(a))^{\xi-1}\tilde{A}_y+\mathcal{I}_{a +}^{\mu\,;\, \varphi}g_y(t),
\end{equation}
where $g_y(\cdot)\in C_{2-\xi;\,\varphi} (J, \R)$ satisfies the functional equation 
\begin{equation}\label{3.13}
g_y(t)=f\left( t, (\varphi(t)-\varphi(a))^{\xi-1}\tilde{A}_y+\mathcal{I}_{a +}^{\mu\,;\, \varphi}g_y(t), g_y(t) \right), \quad t\in J 
\end{equation}
and 
\begin{equation}\label{3.14}
\tilde{A}_y=\frac{1}{\Lambda\Gamma(\xi)}\left[\sum_{i=1}^{m}\lambda_i\mathcal{I}_{a+}^{\mu+\delta_i\,;\,\varphi}g_y(\tau_i)-\mathcal{I}_{a +}^{\mu\,;\, \varphi}g_y(b) \right].
\end{equation}
\end{theorem}
\begin{proof} Assume that  $y\in C_{2-\xi;\,\varphi} (J, \R)$ is the solution of integral equation \eqref{3.12}.
Operating $^H \mathcal{D}^{\mu,\,\nu\,;\, \varphi}_{a +}$ on both sides of \eqref{3.12} and using the Lemma \ref{lema2} and Lemma \ref{teo1}, we obtain
\begin{align*}
^H \mathcal{D}^{\mu,\,\nu\,;\, \varphi}_{a +}y(t)
&=\tilde{A}_y\,^H \mathcal{D}^{\mu,\,\nu\,;\, \varphi}_{a +}(\varphi(t)-\varphi(a))^{\xi-1}+^H\mathcal{D}^{\mu,\,\nu\,;\, \varphi}_{a +}\,\mathcal{I}_{a +}^{\mu\,;\, \varphi}g_y(t)\\
&=g_y(t)\\
&=f\left( t, (\varphi(t)-\varphi(a))^{\xi-1}\tilde{A}_y+\mathcal{I}_{a +}^{\mu\,;\, \varphi}g_y(t), g_y(t) \right)\\
&=f\left( t, y(t), ^H \mathcal{D}^{\mu,\,\nu\,;\, \varphi}_{a +}y(t)\right).
\end{align*}
Thus, $y$ satisfies the equation \eqref{eq1}. The proof of the function 
$y$ given by \eqref{3.12} satisfies the boundary conditions \eqref{eq2} and \eqref{eq3} can be completed similarly as in the proof of Theorem \ref{thm3.1} with $h(t)$ replaced by $g_y(t)$.
\end{proof}
\section{Existence Result}
In this section we derive  existence result for nonlocal implicit BVP \eqref{eq1}-\eqref{eq3}. 
\begin{theorem}\label{th4.1}
Assume the following hypothesis hold:\\
$(H1)$\,$f: J'\times \R  \to \R $ be a continuous function such that $f\left( \cdot, y(\cdot), ^H \mathcal{D}^{\mu,\,\nu\,;\, \varphi}_{a +}y(\cdot)\right)\in C_{2-\xi;\,\varphi} (J, \R)$ for each $y \in C_{2-\xi;\,\varphi} (J, \R)$ that  satisfy the following Lipschitz type condition
$$
\left| f(t, x_1, y_1)-f(t, x_2, y_2)\right| \leq K \left|x_1-x_2 \right| + L \left|y_1-y_2 \right|,\,t\in J,
$$
where $x_1, x_2, y_1, y_2\in \R,\, K>0 $ and $0<L<1.$
Then, nonlocal   BVP for $\varphi$-Hilfer implicit FDEs \eqref{eq1}-\eqref{eq3} has at least one solution, provided 
\begin{align}\label{41}
\sigma&=\frac{K\Gamma(\xi-1)(\varphi(b)-\varphi(a))^{\mu}}{1-L}\left\lbrace\frac{(\varphi(b)-\varphi(a))^{\xi-1}}{\Gamma(\xi)\Lambda}\Omega +\frac{1}{\Gamma(\xi+\mu-1)} \right\rbrace <1,
\end{align}
where 
\begin{align}\label{42}
\Omega&=\sum_{i=1}^{m}
\frac{\lambda_i ( \varphi (b)-\varphi (a))^{\delta_i }}{\Gamma(\xi+\mu+\delta_i-1)}+\frac{1}{\Gamma(\xi+\mu-1)}
\end{align} and $\Lambda$ is defined in equation \eqref{a1}.
\end{theorem}
\begin{proof}
In the view of  Theorem \ref{thm3.2}, the equivalent fractional integral  equation to  the nonlocal   BVP for $\varphi$-Hilfer implicit FDEs \eqref{eq1}-\eqref{eq3} can be written as operator equation as follows 
\begin{equation}\label{4.1}
y=Py+Qy,\, y\in C_{2-\xi;\,\varphi} (J, \R),
\end{equation}
where $P$ and $Q$ are the operators defined  on $C_{2-\xi;\,\varphi} (J, \R)$  by
\begin{align*}
Py(t)&=(\varphi(t)-\varphi(a))^{\xi-1}\tilde{A}_y, \, t\in J\quad\
\text{and}\\
Qy(t)&=\mathcal{I}_{a +}^{\mu\,;\, \varphi}g_y(t)=\frac{1}{\Gamma(\mu)}\int_{a}^{t}\varphi'(s)(\varphi(t)-\varphi(s))^{\mu-1} g_y(s)ds, \, t\in J,
\end{align*}
where $g_y\in C_{2-\xi;\,\varphi} (J, \R)$ satisfies the functional equation 
$$
g_y(t)=f\left( t, (\varphi(t)-\varphi(a))^{\xi-1}\tilde{A}_y+\mathcal{I}_{a +}^{\mu\,;\, \varphi}g_y(t), g_y(t) \right), \quad t\in J,
$$ 
and 
$$
\tilde{A}_y=\frac{1}{\Lambda\Gamma(\xi)}\left[\sum_{i=1}^{m}\lambda_i\,\mathcal{I}_{a+}^{\mu+\delta_i\,;\,\varphi}g_y(\tau_i)-\mathcal{I}_{a +}^{\mu\,;\, \varphi}g_y(b) \right].
$$
To prove the nonlocal implicit BVP \eqref{eq1}-\eqref{eq3} has a solution in $C_{2-\xi;\,\varphi} (J, \R)$  is equivalent to show that the operator equation \eqref{4.1} has a fixed point. Define, 
\begin{align*}
\zeta&=\frac{M\Gamma(\xi-1)(\varphi(t)-\varphi(a))^{\mu}}{1-L}\left\lbrace\frac{(\varphi(t)-\varphi(a))^{\xi-1}}{\Gamma(\xi)}\Omega+\frac{1}{\Gamma(\xi+\mu-1)} \right\rbrace,
\end{align*}
where
$$
M=\underset{t\in J 
}{\max }\left\vert \left( \varphi \left( t\right) -\varphi \left( a\right) \right)
^{2-\xi }f\left( t,\,0,\,0\right) \right\vert.
$$
Choose $r$ such that
$$
r\geq \frac{\zeta}{1-\sigma}$$
and consider the closed ball
$$
B_r=\left\lbrace y\in C_{2-\xi;\,\varphi} (J, \R):\left\Vert y\right\Vert _{C_{2-\xi ;\,\varphi }\left(J,\,\R\right)  } \leq r \right\rbrace. 
$$
To obtain the fixed point of the operator equation \eqref{4.1}, we prove  that the operators $P$ and $Q$ satisfies the conditions of Theorem \ref{Krasnoselskii} (Krasnoselskii fixed point theorem). We give the proof in the following steps:\\
\textbf{Step 1:} $Px+Qy\in B_r$ for all $x,y\in  B_r.$\\
Let any $x,y\in  B_r$. Then,
$
\left( \varphi \left( t\right) -\varphi \left( a\right) \right)^{2-\xi }Px(t)=(\varphi(t)-\varphi(a))\tilde{A}_x\in C(J, \R).
$
This implies, $Px \in C_{2-\xi;\,\varphi} (J, \R)$. 

Since $g_y\left( \cdot\right)=f(\cdot, y(\cdot), g_y(\cdot)) \in C_{2-\xi;\,\varphi} (J, \R) $ for any $y \in C_{2-\xi;\,\varphi} (J, \R), 0\leq 2-\xi<1$ and $Qy(t)=\mathcal{I}_{a +}^{\mu\,;\, \varphi}g_y(t)$, by Lemma \ref{lema1} we have $Qy \in C_{2-\xi;\,\varphi} (J, \R)$. Using  triangle inequality in the space $C_{2-\xi;\,\varphi} (J, \R)  $ we have
\begin{equation}\label{4.2}
\left\Vert Px+Qy\right\Vert _{C_{2-\xi ;\,\varphi }\left(J,\,\R\right)  }\leq \left\Vert Px\right\Vert _{C_{2-\xi ;\,\varphi }\left(J,\,\R\right)  }+\left\Vert Qy\right\Vert _{C_{2-\xi ;\,\varphi }\left(J,\,\R\right)  }.
\end{equation}
Now, by  using hypothesis ($H1$), we have
\begin{small}
\begin{align}\label{4.3}
&\left|  \left( \varphi \left( t\right) -\varphi \left( a\right) \right)^{2-\xi }Px(t)\right|\nonumber\\
&=\left| (\varphi(t)-\varphi(a))\tilde{A}_x \right| \nonumber\\
&=\left| \frac{(\varphi(t)-\varphi(a))}{\Lambda\Gamma(\xi)} \left[\sum_{i=1}^{m}\lambda_i\,\mathcal{I}_{a+}^{\mu+\delta_i\,;\,\varphi}g_x(\tau_i)-\mathcal{I}_{a +}^{\mu\,;\, \varphi}g_x(b) \right]  \right| \nonumber\\
&\leq \frac{(\varphi(t)-\varphi(a))}{\Lambda\Gamma(\xi)} \left[\sum_{i=1}^{m}\frac{\lambda_i}{\Gamma(\mu+\delta_i)}\int_{a}^{\tau_i}\varphi'(s)(\varphi(\tau_i)-\varphi(s))^{\mu+\delta_i-1} \left| f(s, x(s), g_x(s))\right| ds\right.\nonumber\\
&\left.\quad+\frac{1}{\Gamma(\mu)}\int_{a}^{b}\varphi'(s)(\varphi(b)-\varphi(s))^{\mu-1} \left| f(s, x(s), g_x(s))\right|ds\right]\nonumber\\
&\leq \frac{(\varphi(t)-\varphi(a))}{\Lambda\Gamma(\xi)} \left[\sum_{i=1}^{m}\frac{\lambda_i}{\Gamma(\mu+\delta_i)}\int_{a}^{\tau_i}\varphi'(s)(\varphi(\tau_i)-\varphi(s))^{\mu+\delta_i-1} \times \right. \nonumber\\ 
&\left. \qquad \qquad\left\lbrace  \left| f(s, x(s), g_x(s))-f(s,0,0)\right|+\left|f(s,0,0)\right|\right\rbrace  ds\right.\nonumber\\
&\left.\quad+\frac{1}{\Gamma(\mu)}\int_{a}^{b}\varphi'(s)(\varphi(b)-\varphi(s))^{\mu-1}\left\lbrace  \left| f(s, x(s), g_x(s))-f(s,0,0)\right|+\left|f(s,0,0)\right|\right\rbrace ds\right]\nonumber\\
&\leq \frac{(\varphi(t)-\varphi(a))}{\Lambda\Gamma(\xi)} \left[\sum_{i=1}^{m}\frac{\lambda_i}{\Gamma(\mu+\delta_i)}\int_{a}^{\tau_i}\varphi'(s)(\varphi(\tau_i)-\varphi(s))^{\mu+\delta_i-1} \left\lbrace   K\left|  x(s)\right| +L\left|  g_x(s)\right| +\left|f(s,0,0)\right|\right\rbrace  ds\right.\nonumber\\
&\left.\quad+\frac{1}{\Gamma(\mu)}\int_{a}^{b}\varphi'(s)(\varphi(b)-\varphi(s))^{\mu-1}\left\lbrace  K\left|  x(s)\right| +L\left|  g_x(s)\right|+\left|f(s,0,0)\right|\right\rbrace ds\right]
\end{align}
But, by  hypothesis ($H1$),  we have
\begin{align*}
\left|  g_x(t)\right|&\leq\left| f(t, x(t), g_x(t))-f(t,0,0)\right|+\left|f(t,0,0)\right|\\
&\leq  K\left|  x(t)\right| +L\left|  g_x(t)\right|+\left|f(t,0,0)\right|, \, t\in J.
\end{align*}
\end{small}
This gives
\begin{equation}\label{4.4}
\left|  g_x(t)\right|\leq  \frac{K}{1-L}\left|  x(t)\right| +\frac{1}{1-L}\left|f(t,0,0)\right|, \, t\in J.
\end{equation}
Using equation \eqref{4.4} in equation \eqref{4.3}, we get
\begin{small}
\begin{align*}
&\left|  \left( \varphi \left( t\right) -\varphi \left( a\right) \right)^{2-\xi }Px(t)\right|\nonumber\\
&\leq \frac{(\varphi(t)-\varphi(a))}{\Lambda\Gamma(\xi)} \left[\sum_{i=1}^{m}\frac{\lambda_i}{\Gamma(\mu+\delta_i)}\int_{a}^{\tau_i}\varphi'(s)(\varphi(\tau_i)-\varphi(s))^{\mu+\delta_i-1}\right. \\
&\qquad\left.\times\left\lbrace   K\left|  x(s)\right| +L\left\lbrace \frac{K}{1-L}\left|  x(s)\right| +\frac{1}{1-L}\left|f(s,0,0)\right| \right\rbrace +\left|f(s,0,0)\right|\right\rbrace  ds\right.\nonumber\\
&\left.\quad+\frac{1}{\Gamma(\mu)}\int_{a}^{b}\varphi'(s)(\varphi(b)-\varphi(s))^{\mu-1}\left\lbrace  K\left|  x(s)\right| +L\left\lbrace \frac{K}{1-L}\left|  x(s)\right| +\frac{1}{1-L}\left|f(s,0,0)\right| \right\rbrace +\left|f(s,0,0)\right|\right\rbrace ds\right]\\
&\leq \frac{(\varphi(t)-\varphi(a))}{\Lambda\Gamma(\xi)} \left[\sum_{i=1}^{m}\frac{\lambda_i}{\Gamma(\mu+\delta_i)}\int_{a}^{\tau_i}\varphi'(s)(\varphi(\tau_i)-\varphi(s))^{\mu+\delta_i-1}\left\lbrace    \frac{K}{1-L}\left|  x(s)\right| +\frac{1}{1-L}\left|f(s,0,0)\right| \right\rbrace  ds\right.\nonumber\\
&\left.\quad+\frac{1}{\Gamma(\mu)}\int_{a}^{b}\varphi'(s)(\varphi(b)-\varphi(s))^{\mu-1}\left\lbrace    \frac{K}{1-L}\left|  x(s)\right| +\frac{1}{1-L}\left|f(s,0,0)\right| \right\rbrace ds\right]\\
&\leq \frac{(\varphi(t)-\varphi(a))}{\Lambda\Gamma(\xi)} \left[\frac{K}{1-L}\left\lbrace \sum_{i=1}^{m}\frac{\lambda_i}{\Gamma(\mu+\delta_i)}\int_{a}^{\tau_i}\varphi'(s)(\varphi(\tau_i)-\varphi(s))^{\mu+\delta_i-1} (\varphi(s)-\varphi(a))^{\xi-2}\right.\right. \\
&\qquad \left.\left.\times\left| (\varphi(s)-\varphi(a))^{2-\xi} x(s)\right|\right.\right.ds\\
&\left.\left. \qquad+\frac{1}{\Gamma(\mu)}\int_{a}^{b}\varphi'(s)(\varphi(b)-\varphi(s))^{\mu-1}(\varphi(s)-\varphi(a))^{\xi-2} \left| (\varphi(s)-\varphi(a))^{2-\xi} x(s)\right|  ds \right\rbrace \right.\nonumber\\
&\left.\qquad+\frac{1}{1-L}\left\lbrace \sum_{i=1}^{m}\frac{\lambda_i}{\Gamma(\mu+\delta_i)}\int_{a}^{\tau_i}\varphi'(s)(\varphi(\tau_i)-\varphi(s))^{\mu+\delta_i-1} (\varphi(s)-\varphi(a))^{\xi-2}\left| (\varphi(s)-\varphi(a))^{2-\xi} f(s,0,0)\right|\right.\right.ds\\
&\left.\left. \qquad+\frac{1}{\Gamma(\mu)}\int_{a}^{b}\varphi'(s)(\varphi(b)-\varphi(s))^{\mu-1}(\varphi(s)-\varphi(a))^{\xi-2} \left| (\varphi(s)-\varphi(a))^{2-\xi}f(s,0,0)\right|  ds \right\rbrace \right]\\
&\leq \frac{(\varphi(t)-\varphi(a))}{\Lambda\Gamma(\xi)} \left[\frac{K}{1-L}\left\Vert x\right\Vert _{C_{2-\xi ;\,\varphi }\left(J,\,\R\right)  }\left\lbrace \sum_{i=1}^{m}\frac{\lambda_i}{\Gamma(\mu+\delta_i)}\int_{a}^{b}\varphi'(s)(\varphi(b)-\varphi(s))^{\mu+\delta_i-1} (\varphi(s)-\varphi(a))^{\xi-2}ds\right.\right.\\
&\left.\left. \qquad+\frac{1}{\Gamma(\mu)}\int_{a}^{b}\varphi'(s)(\varphi(b)-\varphi(s))^{\mu-1}(\varphi(s)-\varphi(a))^{\xi-2} ds \right\rbrace  \right.\nonumber\\
&\left.\qquad+\frac{M}{1-L}\left\lbrace \sum_{i=1}^{m}\frac{\lambda_i}{\Gamma(\mu+\delta_i)}\int_{a}^{b}\varphi'(s)(\varphi(b)-\varphi(s))^{\mu+\delta_i-1} (\varphi(s)-\varphi(a))^{\xi-2}\right.\right.ds\\
&\left.\left. \qquad+\frac{1}{\Gamma(\mu)}\int_{a}^{b}\varphi'(s)(\varphi(b)-\varphi(s))^{\mu-1}(\varphi(s)-\varphi(a))^{\xi-2} ds \right\rbrace  \right]\\
&\leq \frac{(\varphi(t)-\varphi(a))}{\Lambda\Gamma(\xi)} \left[\frac{K}{1-L}\left\Vert x\right\Vert _{C_{2-\xi ;\,\varphi }\left(J,\,\R\right)  }\left\lbrace \sum_{i=1}^{m}\frac{\lambda_i\Gamma(\xi-1)}{\Gamma(\mu+\delta_i+\xi-1)} (\varphi(b)-\varphi(a))^{\mu+\delta_i+\xi-2}\right.\right.\\
&\left.\left. \qquad+\frac{\Gamma(\xi-1)}{\Gamma(\mu+\xi-1)}(\varphi(b)-\varphi(a))^{\mu+\xi-2}  \right\rbrace  \right.\nonumber\\
&\left.\qquad+\frac{M}{1-L}\left\lbrace \sum_{i=1}^{m}\frac{\lambda_i\Gamma(\xi-1)}{\Gamma(\mu+\delta_i+\xi-1)} (\varphi(b)-\varphi(a))^{\mu+\delta_i+\xi-2}+\frac{\Gamma(\xi-1)}{\Gamma(\mu+\xi-1)}(\varphi(b)-\varphi(a))^{\mu+\xi-2}\right\rbrace    \right].
\end{align*}
\end{small}
Therefore,
\begin{small}
\begin{align}\label{4.5}
&\left\Vert Px\right\Vert _{C_{2-\xi ;\,\varphi }\left(J,\,\R\right)  }\nonumber\\
&=\underset{t\in J 
}{\max }\left\vert \left( \varphi \left( t\right) -\varphi \left( a\right) \right)
^{2-\xi }Px\left( t\right) \right\vert\nonumber\\
&\leq \frac{(\varphi(b)-\varphi(a))^{\mu+\xi-1}\Gamma(\xi-1)}{\Lambda\Gamma(\xi)} \left[\frac{K\,r}{1-L} \left\lbrace \sum_{i=1}^{m}\frac{\lambda_i}{\Gamma(\mu+\delta_i+\xi-1)} (\varphi(b)-\varphi(a))^{\delta_i}+\frac{1}{\Gamma(\mu+\xi-1)}\right\rbrace  \right.\nonumber\\
&\left.\qquad+\frac{M}{1-L}\left\lbrace \sum_{i=1}^{m}\frac{\lambda_i}{\Gamma(\mu+\delta_i+\xi-1)} (\varphi(b)-\varphi(a))^{\delta_i}+\frac{1}{\Gamma(\mu+\xi-1)}\right\rbrace\right]\nonumber\\
&\leq \frac{(\varphi(b)-\varphi(a))^{\mu+\xi-1}\Gamma(\xi-1)}{\Lambda\Gamma(\xi)} \frac{(K\,r+M)}{1-L}\Omega=A.
\end{align}
Further, by using  ($H1$) and  inequality \eqref{4.4}, we have
\begin{align*}
&\left|  \left( \varphi \left( t\right) -\varphi \left( a\right) \right)^{2-\xi }Qy(t)\right|\nonumber\\
&=\left| \frac{ \left( \varphi \left( t\right) -\varphi \left( a\right) \right)^{2-\xi }}{\Gamma(\mu)}\int_{a}^{t}\varphi'(s)(\varphi(t)-\varphi(s))^{\mu-1} g_y(s)ds \right| \nonumber\\
&\leq \frac{ \left( \varphi \left( t\right) -\varphi \left( a\right) \right)^{2-\xi }}{\Gamma(\mu)}\int_{a}^{t}\varphi'(s)(\varphi(t)-\varphi(s))^{\mu-1}\left| f(s, y(s), g_y(s))\right|ds\nonumber\\
&\leq \frac{ \left( \varphi \left( t\right) -\varphi \left( a\right) \right)^{2-\xi }}{\Gamma(\mu)}\int_{a}^{t}\varphi'(s)(\varphi(t)-\varphi(s))^{\mu-1}\left\lbrace  \left| f(s, y(s), g_y(s))-f(s,0,0)\right|+\left|f(s,0,0)\right|\right\rbrace ds\nonumber\\
&\leq \frac{ \left( \varphi \left( t\right) -\varphi \left( a\right) \right)^{2-\xi }}{\Gamma(\mu)}\int_{a}^{t}\varphi'(s)(\varphi(t)-\varphi(s))^{\mu-1}\left\lbrace  K\left|  y(s)\right| +L\left|  g_y(s)\right|+\left|f(s,0,0)\right|\right\rbrace ds\nonumber\\
&\leq\frac{ \left( \varphi \left( t\right) -\varphi \left( a\right) \right)^{2-\xi }}{\Gamma(\mu)}\int_{a}^{t}\varphi'(s)(\varphi(t)-\varphi(s))^{\mu-1}\nonumber\\
&\qquad\times\left\lbrace  K\left|  y(s)\right| +L\left\lbrace \frac{K}{1-L}\left|  y(s)\right| +\frac{1}{1-L}\left|f(s,0,0)\right| \right\rbrace +\left|f(s,0,0)\right|\right\rbrace ds\nonumber\\
&\leq \frac{ \left( \varphi \left( t\right) -\varphi \left( a\right) \right)^{2-\xi }}{\Gamma(\mu)}\int_{a}^{t}\varphi'(s)(\varphi(t)-\varphi(s))^{\mu-1}\left\lbrace    \frac{K}{1-L}\left| y(s)\right| +\frac{1}{1-L}\left|f(s,0,0)\right| \right\rbrace ds\nonumber\\
&\leq  \left( \varphi \left( t\right) -\varphi \left( a\right) \right)^{2-\xi }\left\lbrace \frac{K}{(1-L)\Gamma(\mu)}\int_{a}^{t}\varphi'(s)(\varphi(t)-\varphi(s))^{\mu-1}(\varphi(s)-\varphi(a))^{\xi-2} \left| (\varphi(s)-\varphi(a))^{2-\xi} y(s)\right|  ds  \right.\nonumber\\
&\left.\qquad+\frac{1}{(1-L)\Gamma(\mu)}\int_{a}^{t}\varphi'(s)(\varphi(t)-\varphi(s))^{\mu-1} (\varphi(s)-\varphi(a))^{\xi-2}\left| (\varphi(s)-\varphi(a))^{2-\xi} f(s,0,0)\right|ds\right\rbrace \nonumber\\
&\leq  \left( \varphi \left( t\right) -\varphi \left( a\right) \right)^{2-\xi }\left\lbrace \frac{K}{1-L}\left\Vert y\right\Vert _{C_{2-\xi ;\,\varphi }\left(J,\,\R\right)  }\frac{1}{\Gamma(\mu)}\int_{a}^{t}\varphi'(s)(\varphi(t)-\varphi(s))^{\mu-1}(\varphi(s)-\varphi(a))^{\xi-2} ds\right.\nonumber\\
&\left.\qquad+\frac{M}{1-L}\frac{1}{\Gamma(\mu)}\int_{a}^{t}\varphi'(s)(\varphi(t)-\varphi(s))^{\mu-1}(\varphi(s)-\varphi(a))^{\xi-2} ds \right\rbrace \nonumber\\
&\leq  \left( \varphi \left( t\right) -\varphi \left( a\right) \right)^{2-\xi }\left\lbrace \frac{K}{1-L}\left\Vert y\right\Vert _{C_{2-\xi ;\,\varphi }\left(J,\,\R\right)  }\frac{\Gamma(\xi-1)}{\Gamma(\mu+\xi-1)}(\varphi(t)-\varphi(a))^{\mu+\xi-2}\right.\nonumber\\
&\left.\qquad+\frac{M}{1-L}\frac{\Gamma(\xi-1)}{\Gamma(\mu+\xi-1)}(\varphi(t)-\varphi(a))^{\mu+\xi-2}\right\rbrace\nonumber\\
&\leq  \frac{(\varphi(t)-\varphi(a))^{\mu}\Gamma(\xi-1)}{(1-L)\Gamma(\mu+\xi-1)} \left\lbrace K\left\Vert y\right\Vert _{C_{2-\xi ;\,\varphi }\left(J,\,\R\right) }+ M\right\rbrace.
\end{align*}
\end{small}
Therefore,
\begin{align}\label{4.6}
\left\Vert Qy\right\Vert _{C_{2-\xi ;\,\varphi }\left(J,\,\R\right)  }
&=\underset{t\in J 
}{\max }\left\vert \left( \varphi \left( t\right) -\varphi \left( a\right) \right)
^{2-\xi }Qy\left( t\right) \right\vert\nonumber\\
&\leq \frac{(\varphi(b)-\varphi(a))^{\mu}\Gamma(\xi-1)}{(1-L)\Gamma(\mu+\xi-1)} \left\lbrace Kr+ M\right\rbrace=B.
\end{align}
From inequalities \eqref{4.5} and \eqref{4.6} and the definition of $r$ given in \eqref{4.2}, we obtain
$$
\left\Vert Px+Qy\right\Vert _{C_{2-\xi ;\,\varphi }\left(J,\,\R\right)  }\leq A+B=r\sigma+\zeta\leq r.
$$
This proves, $Px+Qy\in B_r$ for all $x,y\in  B_r.$\\
\textbf{Step 2:} Operator $P$ is contraction on $B_r$.\\
Let any $y_1, y_2 \in  B_r$ and $t\in  J$. Then 
\begin{align}\label{4.7}
&\left|  \left( \varphi \left( t\right) -\varphi \left( a\right) \right)^{2-\xi }\left( Py_1(t)-Py_2(t)\right) \right|\nonumber\\
&\leq \frac{(\varphi(t)-\varphi(a))}{\Lambda\Gamma(\xi)} \left[\sum_{i=1}^{m}\frac{\lambda_i}{\Gamma(\mu+\delta_i)}\int_{a}^{\tau_i}\varphi'(s)(\varphi(\tau_i)-\varphi(s))^{\mu+\delta_i-1} \left| g_{y_1}(s)-g_{y_2}(s)\right| ds\right.\nonumber\\
&\left.\quad+\frac{1}{\Gamma(\mu)}\int_{a}^{b}\varphi'(s)(\varphi(b)-\varphi(s))^{\mu-1} \left| g_{y_{1}}(s)-g_{y_{2}}(s)\right|ds\right].
\end{align}
But by hypothesis ($H1$), we have
\begin{align*}
 \left| g_{y_{1}}(t)-g_{y_{2}}(t)\right|&= \left| f(t,y_{1}(t), g_{y_{1}}(t))-f(t,y_{2}(t), g_{y_{2}}(t))\right|\\
 &\leq K \left| y_{1}(t)-y_{2}(t)\right|+ L\left| g_{y_{1}}(t)-g_{y_{2}}(t)\right|, \, t\in J.
\end{align*}
Thus,
\begin{equation}\label{4.8}
\left| g_{y_{1}}(t)-g_{y_{2}}(t)\right|\leq \frac{K}{1-L} \left| y_{1}(t)-y_{2}(t)\right|, \, t\in J.
\end{equation}
Using equation \eqref{4.8} in equation \eqref{4.7} we have
\begin{align*}
&\left|  \left( \varphi \left( t\right) -\varphi \left( a\right) \right)^{2-\xi }\left( Py_1(t)-Py_2(t)\right) \right|\nonumber\\
& \leq  \frac{K}{1-L} \frac{(\varphi(t)-\varphi(a))}{\Lambda\Gamma(\xi)} \left[\sum_{i=1}^{m}\frac{\lambda_i}{\Gamma(\mu+\delta_i)}\int_{a}^{\tau_i}\varphi'(s)(\varphi(\tau_i)-\varphi(s))^{\mu+\delta_i-1} \right.\nonumber \\
& \left. \qquad\times\left( \varphi \left( s\right) -\varphi \left( a\right) \right)^{\xi-2 }\left|\left( \varphi \left( s\right) -\varphi \left( a\right) \right)^{2-\xi } (y_{1}(s)-y_{2}(s))\right| ds\right.\nonumber\\
&\left.\quad+\frac{1}{\Gamma(\mu)}\int_{a}^{b}\varphi'(s)(\varphi(b)-\varphi(s))^{\mu-1} \left( \varphi \left( s\right) -\varphi \left( a\right) \right)^{\xi-2 }\left|\left( \varphi \left( s\right) -\varphi \left( a\right) \right)^{2-\xi } (y_{1}(s)-y_{2}(s))\right|ds\right]\nonumber\\
& \leq  \frac{K}{1-L} \frac{(\varphi(t)-\varphi(a))}{\Lambda\Gamma(\xi)}\left\Vert y_1-y_2\right\Vert _{C_{2-\xi ;\,\varphi }\left(J,\,\R\right)  }\times\\
&\qquad \left[\sum_{i=1}^{m}\frac{\lambda_i\Gamma(\xi-1)}{\Gamma(\mu+\delta_i+\xi-1)} (\varphi(b)-\varphi(a))^{\mu+\delta_i+\xi-2}+\frac{\Gamma(\xi-1)}{\Gamma(\mu+\xi-1)}(\varphi(b)-\varphi(a))^{\mu+\xi-2}\right].
\end{align*}
Therefore,
\begin{align*}
\left\Vert Py_1-Py_2\right\Vert _{C_{2-\xi ;\,\varphi }\left(J,\,\R\right)  }
&=\underset{t\in J 
}{\max }\left\vert \left( \varphi \left( t\right) -\varphi \left( a\right) \right)^{2-\xi }\left( Py_1(t)-Py_2(t)\right) \right\vert\nonumber\\
&\leq \left\Vert y_1-y_2\right\Vert _{C_{2-\xi ;\,\varphi }\left(J,\,\R\right)  } \frac{K\Gamma(\xi-1)}{1-L} \frac{(\varphi(t)-\varphi(a))^{\mu+\xi-1}}{\Lambda\Gamma(\xi)}\Omega\nonumber\\
&\leq \sigma \left\Vert y_1-y_2\right\Vert _{C_{2-\xi ;\,\varphi }\left(J
,\,\R\right)  }.
\end{align*}
Since $\sigma <1$, $P$ is contraction on $B_r.$

\textbf{Step 3:} The operator $Q$ is compact and continuous.\\
We have already proved that, $Q$ is self mapping on $C_{2-\xi ;\,\varphi }\left(J,\,\R\right).$ Further, from Step 1 we have 
$$
\left\Vert Qy\right\Vert _{C_{2-\xi ;\,\varphi }\left(J,\,\R\right)  }
\leq \frac{(\varphi(b)-\varphi(a))^{\mu}\Gamma(\xi-1)}{(1-L)\Gamma(\mu+\xi-1)} \left\lbrace Kr+ M\right\rbrace,\, y \in B_r. 
$$
Therefore, $Q(B_r)=\left\lbrace Qy: y\in B_r \right\rbrace $ is uniformly bounded. Next, we prove that $Q(B_r)$ is equicontinuous. Let any $y\in B_r$ and $t_1,\,t_2\in J$ with $t_1>t_2$. Then
\begin{align*}
&\left| Qy(t_1)-Qy(t_2)\right|\\
&=\left| \frac{ 1}{\Gamma(\mu)}\int_{a}^{t_1}\varphi'(s)(\varphi(t_1)-\varphi(s))^{\mu-1} g_y(s)ds -\frac{ 1}{\Gamma(\mu)}\int_{a}^{t_2}\varphi'(s)(\varphi(t_2)-\varphi(s))^{\mu-1} g_y(s)ds \right| \\
&\leq\left| \frac{ 1}{\Gamma(\mu)}\int_{a}^{t_1}\varphi'(s)(\varphi(t_1)-\varphi(s))^{\mu-1}(\varphi(s)-\varphi(a))^{\xi-2} \left| (\varphi(s)-\varphi(a))^{2-\xi} f(s, y(s), g_y(s))\right|ds\right.\\
&\left.~-\frac{ 1}{\Gamma(\mu)}\int_{a}^{t_2}\varphi'(s)(\varphi(t_2)-\varphi(s))^{\mu-1}(\varphi(s)-\varphi(a))^{\xi-2} \left| (\varphi(s)-\varphi(a))^{2-\xi} f(s, y(s), g_y(s))\right|ds \right|.  
\end{align*}
Since $g_y\left( \cdot\right)=f(\cdot, y(\cdot), g_y(\cdot)) \in C_{2-\xi;\,\varphi} (J, \R) $ for any $y \in C_{2-\xi;\,\varphi} (J, \R)$, $(\varphi(\cdot)-\varphi(a))^{2-\xi} f(\cdot, y(\cdot), g_y(\cdot))$ is continuous on $J=[a,\,b]$. Therefore, there exists $\mathfrak{K}\in \R$ such that 
$$
\left|(\varphi(t)-\varphi(a))^{2-\xi} f(t, y(t), g_y(t)) \right| \leq  \mathfrak{K},~ \text{for all}~ t\in J.
$$
Thus,
\begin{align*}
&\left| Qy(t_1)-Qy(t_2)\right|\\
&\leq\left| \frac{  \mathfrak{K}}{\Gamma(\mu)}\int_{a}^{t_1}\varphi'(s)(\varphi(t_1)-\varphi(s))^{\mu-1}(\varphi(s)-\varphi(a))^{\xi-2} ds\right.\\
&\left.~-\frac{  \mathfrak{K}}{\Gamma(\mu)}\int_{a}^{t_2}\varphi'(s)(\varphi(t_2)-\varphi(s))^{\mu-1}(\varphi(s)-\varphi(a))^{\xi-2} ds \right|\\
&\leq\left|\mathfrak{K} \frac{\Gamma(\xi-1)}{\Gamma(\mu+\xi-1)}(\varphi(t_1)-\varphi(a))^{\mu+\xi-2}-\mathfrak{K}\frac{\Gamma(\xi-1)}{\Gamma(\mu+\xi-1)}(\varphi(t_2)-\varphi(a))^{\mu+\xi-2} \right|.
\end{align*}
Hence, 
$$
\left| Qy(t_1)-Qy(t_2)\right|\leq \mathfrak{K} \frac{\Gamma(\xi-1)}{\Gamma(\mu+\xi-1)}\left|(\varphi(t_1)-\varphi(a))^{\mu+\xi-2}-(\varphi(t_2)-\varphi(a))^{\mu+\xi-2} \right|.
$$
Using continuity of  $\varphi$, we observe that $\left| Qy(t_1)-Qy(t_2)\right|\to 0$ as $\left|t_1-t_2 \right|\to 0 $. This proves  $Q(B_r)$ is equicontinuous. Thus, by Arzel$\grave{a}$-Ascoli Theorem, $Q(B_r)$ is relatively compact. We proved that, operator $Q$ is compact.

Since $P$ and $Q$ satisfy all the conditions of Krasnoselskii fixed point theorem (Theorem \ref{Krasnoselskii}), the operator equation defined in the equation \eqref{4.1}, has at least  one fixed point in the space $C_{2-\xi ;\,\varphi }\left(J,\,\R\right)$ which is the solution of the nonlocal BVP for implicit FDEs \eqref{eq1}-\eqref{eq3}.

\end{proof}

\section{Ulam Stability Results}
We begin with definitions regarding  Ulam stabilities.  To discuss the these stabilities for the problem  \eqref{eq1}, we follow the approach of \cite{jose3}.

\begin{definition}\label{d1}
The equation \eqref{eq1} is Ulam-Hyers stable if there exists a real number $C_f > 0 $, such that for each $\epsilon > 0 $ and for each solution $z \in C_{2-\xi ;\, \varphi }(J, \R)$   of  the inequality
\begin{equation}\label{ed1}
\left| ^H \mathcal{D}^{\mu,\,\nu\,;\, \varphi}_{a +}z(t)
- f(t, z(t),^H \mathcal{D}^{\mu,\,\nu\,;\, \varphi}_{a +}z(t))\right| \leq \epsilon, \,\, t\in J,
\end{equation} 
there exists solution $y \in C_{2-\xi ;\, \varphi }(J, \R)$   of  the equation \eqref{eq1} with
$$
\left\Vert z-y\right\Vert _{C_{2-\xi ;\,\varphi}\left(J,\,\R\right)  }\leq C_f \,\epsilon,\quad t\in J.
$$
\end{definition}

\begin{definition}\label{d2}
The equation \eqref{eq1}  is generalized Ulam-Hyers stable if there exists a function $\phi_f \in C(\R_+,\R_+)$, $\phi_f (0) = 0$, such that for  each solution $z \in C_{2-\xi ;\, \varphi }(J, \R)$   of  the inequality \eqref{ed1}
there exists solution $y \in C_{2-\xi ;\, \varphi }(J, \R)$   of  the equation \eqref{eq1} with
$$
\left\Vert z-y\right\Vert _{C_{2-\xi ;\,\varphi}\left(J,\,\R\right)  }\leq \phi_f (\epsilon),\quad t\in J.
$$
\end{definition}

\begin{definition}\label{d3}
The equation \eqref{eq1} is Ulam-Hyers-Rassias stable, with respect to the  $\chi\in C(J,\,\R_+)$,  if there exists a  $C_{f,\chi} > 0 $ such that for each $\epsilon > 0 $ and for each solution $z \in C_{2-\xi ;\, \varphi }(J, \R)$   of  the inequality
\begin{equation}\label{ed3}
\left| ^H \mathcal{D}^{\mu,\,\nu\,;\, \varphi}_{a +}z(t)
- f(t, z(t),^H \mathcal{D}^{\mu,\,\nu\,;\, \varphi}_{a +}z(t))\right| \leq \epsilon\,\chi(t), \quad t\in J,
\end{equation} 
there exists solution $y \in C_{2-\xi ;\, \varphi }(J, \R)$   of  the equation \eqref{eq1} with
$$
\left|  \left(\varphi(t) -\varphi(a) \right) ^{2-\xi}\left( z(t)-y(t)\right) \right| \leq  \,\epsilon\,C_{f,\chi} \,\chi(t),\quad t\in J.
$$
\end{definition}

\begin{definition}\label{d4}
The equation \eqref{eq1} is generalized Ulam-Hyers-Rassias stable, with respect to the  $\chi\in C(J,\,\R_+)$,  if there exists a  $C_{f,\chi} > 0 $ such that for each solution $z \in C_{2-\xi ;\, \varphi }(J, \R)$   of  the inequality
\begin{equation}\label{ed4}
\left| ^H \mathcal{D}^{\mu,\,\nu\,;\, \varphi}_{a +}z(t)
- f(t, z(t),^H \mathcal{D}^{\mu,\,\nu\,;\, \varphi}_{a +}z(t))\right| \leq \,\chi(t), \quad t\in J,
\end{equation} 
there exists solution $y \in C_{2-\xi ;\, \varphi }(J, \R)$   of  the equation \eqref{eq1} with
$$
\left|  \left(\varphi(t) -\varphi(a) \right) ^{2-\xi}\left( z(t)-y(t)\right) \right| \leq  \,C_{f,\chi} \,\chi(t),\quad t\in J.
$$
\end{definition}

\begin{rem}
A function $ z \in C_{2-\xi ;\, \varphi }(J, \R)$  is solution of  the  inequality \eqref{ed1} if and only
if there exists a function $w \in C_{2-\xi ;\, \varphi }(J, \R)$  (which depend on $z$) such that
\begin{enumerate}[topsep=0pt,itemsep=-1ex,partopsep=1ex,parsep=1ex]
\item $ \left|w(t) \right| \leq \epsilon$.
\item  $^H \mathcal{D}^{\mu,\,\nu\,;\, \varphi}_{a +}z(t)
=f(t, z(t),^H \mathcal{D}^{\mu,\,\nu\,;\, \varphi}_{a +}z(t))+w(t)$.
\end{enumerate}
\end{rem}

\begin{theorem}[Ulam-Hyers stability]\label{th6.1}
Assume that the  hypothesis $(H1)$  hold. Then, the equation  \eqref{eq1} is Ulam-Hyers stable provided the condition \eqref{41} hold.
\end{theorem}
\begin{proof}
Let any $\epsilon>0$. Let $z \in C_{2-\xi ;\, \varphi }(J, \R)$ be any solution of  the inequality 
\begin{equation}\label{5.1a}
\left| ^H \mathcal{D}^{\mu,\,\nu\,;\, \varphi}_{a +}z(t)
- f(t, z(t),^H \mathcal{D}^{\mu,\,\nu\,;\, \varphi}_{a +}z(t))\right| \leq \epsilon, \quad t\in J.
\end{equation}
Then, there exists $ w\in C(J,\,\R)$ such that
\begin{equation}\label{5.7}
^H \mathcal{D}^{\mu,\,\nu\,;\, \varphi}_{a +}z(t)
= f(t, z(t),^H \mathcal{D}^{\mu,\,\nu\,;\, \varphi}_{a +}z(t))+w(t),
\end{equation}
and $\left| w(t)\right|\leq \epsilon,\, t\in J.$  In the view of  the Theorem \ref{thm3.2}.
\begin{equation}\label{59}
z(t)=(\varphi(t)-\varphi(a))^{\xi-1}\tilde{A}_z+\mathcal{I}_{a +}^{\mu\,;\, \varphi}g_z(t)+\mathcal{I}_{a +}^{\mu\,;\, \varphi}w(t)
\end{equation}
is the solution of the equation \eqref{5.7}, where $g_z(\cdot)\in C_{2-\xi;\,\varphi} (J, \R)$ satisfies the functional equation 
$$
g_z(t)=f\left( t, (\varphi(t)-\varphi(a))^{\xi-1}\tilde{A}_z+\mathcal{I}_{a +}^{\mu\,;\, \varphi}g_z(t)+\mathcal{I}_{a +}^{\mu\,;\, \varphi}w(t), g_z(t) \right), \quad t\in J,
$$
and 
$$
\tilde{A}_z
=\frac{1}{\Lambda\Gamma(\xi)}\left[\sum_{i=1}^{m}\lambda_i\,\mathcal{I}_{a+}^{\mu+\delta_i\,;\,\varphi}g_z(\tau_i)-\mathcal{I}_{a +}^{\mu\,;\, \varphi}g_z(b) \right].
$$
From \eqref{59}, we have
\begin{equation}\label{60}
\left|z(t)-(\varphi(t)-\varphi(a))^{\xi-1}\tilde{A}_z- \mathcal{I}_{a +}^{\mu\,;\, \varphi}g_z(t)\right| 
\leq \mathcal{I}_{a +}^{\mu\,;\, \varphi}\left| w(t)\right| 
\leq \epsilon \,\,\frac{(\varphi(t)-\varphi(a))^{\mu}}{\Gamma(\mu+1)}.
\end{equation}
Let $y \in C_{2-\xi ;\, \varphi }(J, \R)$ be   solution  of  the problem
\begin{equation}\label{5.2a1}
\begin{cases}
 ^H \mathcal{D}^{\mu,\,\nu\,;\, \varphi}_{a +}y(t)
=f(t, y(t),^H \mathcal{D}^{\mu,\,\nu\,;\, \varphi}_{a +}y(t)),\\
y(a)=z(a),\qquad y(b)=z(b),
\end{cases}
\end{equation}
where $y(b)=\sum_{i=1}^{m}
\lambda_i\,\mathcal{I}_{a +}
^{\delta_i\,;\, \varphi}y(\tau_i)$ and $z(b)=\sum_{i=1}^{m}
\lambda_i\,\mathcal{I}_{a +}
^{\delta_i\,;\, \varphi}z(\tau_i)$.
By Theorem \ref{thm3.2},  the  equivalent fractional integral equation to \eqref{5.2a1} is
\begin{equation}\label{5.3}
y(t)=(\varphi(t)-\varphi(a))^{\xi-1}\tilde{A}_y+\mathcal{I}_{a +}^{\mu\,;\, \varphi}g_y(t),
\end{equation}
where, $g_y$ satisfies functional equation
 \begin{equation*}
g_y(t)=f\left( t, (\varphi(t)-\varphi(a))^{\xi-1}\tilde{A}_y+\mathcal{I}_{a +}^{\mu\,;\, \varphi}g_y(t), g_y(t) \right), \quad t\in J,
\end{equation*} and $$
\tilde{A}_y
=\frac{1}{\Lambda\Gamma(\xi)}\left[\sum_{i=1}^{m}\lambda_i\,\mathcal{I}_{a+}^{\mu+\delta_i\,;\,\varphi}g_y(\tau_i)-\mathcal{I}_{a +}^{\mu\,;\, \varphi}g_y(b) \right].
$$
Now,
\begin{align}\label{5.5}
&\left| \tilde{A}_y-\tilde{A}_z\right|\nonumber\\
&=\left| \left\lbrace \frac{1}{\Lambda\Gamma(\xi)} \left[\sum_{i=1}^{m}\frac{\lambda_i}{\Gamma(\mu+\delta_i)}\int_{a}^{\tau_i}\varphi'(s)(\varphi(\tau_i)-\varphi(s))^{\mu+\delta_i-1} g_y(s)ds\right.\right.\right.\nonumber\\
&\left.\left.\left.\quad-\frac{1}{\Gamma(\mu)}\int_{a}^{b}\varphi'(s)(\varphi(b)-\varphi(s))^{\mu-1} g_y(s)ds\right]\right.\right\rbrace\nonumber\\
&-\left.\left\lbrace \frac{1}{\Lambda\Gamma(\xi)} \left[\sum_{i=1}^{m}\frac{\lambda_i}{\Gamma(\mu+\delta_i)}\int_{a}^{\tau_i}\varphi'(s)(\varphi(\tau_i)-\varphi(s))^{\mu+\delta_i-1} g_z(s)ds\right.\right.\right.\nonumber\\
&\left.\left.\left.\quad-\frac{1}{\Gamma(\mu)}\int_{a}^{b}\varphi'(s)(\varphi(b)-\varphi(s))^{\mu-1} g_z(s)ds\right]\right\rbrace\right|\nonumber \\
&\leq \frac{1}{\Lambda\Gamma(\xi)} \left\lbrace \sum_{i=1}^{m}\frac{\lambda_i}{\Gamma(\mu+\delta_i)}\int_{a}^{\tau_i}\varphi'(s)(\varphi(\tau_i)-\varphi(s))^{\mu+\delta_i-1}\left|g_y(s) - g_z(s)\right| ds\right.\nonumber\\
&\left.\quad-\frac{1}{\Gamma(\mu)}\int_{a}^{b}\varphi'(s)(\varphi(b)-\varphi(s))^{\mu-1} \left|g_y(s) - g_z(s)\right|ds\right\rbrace.
\end{align}
By hypothesis ($H1$), we obtain
\begin{align*}
 \left| g_y(t)-g_z(t)\right|&= \left| f(t,y(t), g_y(t))-f(t,z(t), g_z(t))\right|\\
 &\leq K \left| y(t)-z(t)\right|+ L\left| g_y(t)-g_z(t)\right|.
\end{align*}
This gives,
\begin{equation}\label{5.6}
\left| g_{y}(t)-g_{z}(t)\right|\leq \frac{K}{1-L} \left| y(t)-z(t)\right|.
\end{equation}
Using equation \eqref{5.6} in equation \eqref{5.5} we get
\begin{align}\label{61}
&\left| \tilde{A}_y-\tilde{A}_z\right|\nonumber\\
&\leq \frac{K}{1-L}\frac{1}{\Lambda\Gamma(\xi)} \left\lbrace \sum_{i=1}^{m}\frac{\lambda_i}{\Gamma(\mu+\delta_i)}\int_{a}^{\tau_i}\varphi'(s)(\varphi(\tau_i)-\varphi(s))^{\mu+\delta_i-1}\left|y(s) - z(s)\right| ds\right.\nonumber\\
&\left.\quad-\frac{1}{\Gamma(\mu)}\int_{a}^{b}\varphi'(s)(\varphi(b)-\varphi(s))^{\mu-1} \left|y(s) - z(s)\right|ds\right\rbrace\nonumber\\
&\leq \frac{K}{1-L}\frac{1}{\Lambda\Gamma(\xi)} \left\lbrace \sum_{i=1}^{m}\lambda_i\,\mathcal{I}_{a+}^{\mu+\delta_i\,;\,\varphi}\left|y(\tau_i) - z(\tau_i)\right| +\mathcal{I}_{a+}^{\mu\,;\,\varphi}\left|y(b) - z(b)\right| \right\rbrace.
\end{align}
Since $y(b)=z(b)$, we must have $y(\tau_i)=z(\tau_i)(i=1,2,\cdots,m)$. Therefore, from inequality \eqref{61}, we obtain $ \tilde{A}_y=\tilde{A}_z$.

Using  equations \eqref{60} and \eqref{5.6}, we have
\begin{align*}
&\left|z(t)-y(t)\right|\\
&=\left|z(t)-\left[ (\varphi(t)-\varphi(a))^{\xi-1}\tilde{A}_y+
 \frac{1}{\Gamma(\mu)}\int_{a}^{t}\varphi'(s)(\varphi(t)-\varphi(s))^{\mu-1}g_y(s) ds\right] \right|\\
&\leq\left|z(t)- (\varphi(t)-\varphi(a))^{\xi-1}\tilde{A}_z- \frac{1}{\Gamma(\mu)}\int_{a}^{t}\varphi'(s)(\varphi(t)-\varphi(s))^{\mu-1}g_z(s) ds\right| \\
&~~+\left| \frac{1}{\Gamma(\mu)}\int_{a}^{t}\varphi'(s)(\varphi(t)-\varphi(s))^{\mu-1}g_z(s) ds-\frac{1}{\Gamma(\mu)}\int_{a}^{t}\varphi'(s)(\varphi(t)-\varphi(s))^{\mu-1}g_y(s) ds \right|\\
&\leq \epsilon \,\,\frac{(\varphi(b)-\varphi(a))^{\mu}}{\Gamma(\mu+1)}+\frac{1}{\Gamma(\mu)}\int_{a}^{t}\varphi'(s)(\varphi(t)-\varphi(s))^{\mu-1}\left|g_z(s)-g_y(s) \right| ds\\
&\leq \epsilon \,\,\frac{(\varphi(b)-\varphi(a))^{\mu}}{\Gamma(\mu+1)}+\frac{K}{1-L}\frac{1}{\Gamma(\mu)}\int_{a}^{t}\varphi'(s)(\varphi(t)-\varphi(s))^{\mu-1}\left|z(s)-y(s) \right|ds.
\end{align*}
Applying Gronwall inequality(Theorem \ref{lema4}) with $u(t)=\left|z(t)-y(t) \right|,\, v(t)=\epsilon \,\,\frac{(\varphi(b)-\varphi(a))^{\mu}}{\Gamma(\mu+1)}$ and $g(t)=\frac{K}{(1-L)\Gamma(\mu)}$, we obtain
$$
\left|z(t)-y(t)\right|
\leq\epsilon\,\,\frac{(\varphi(b)-\varphi(a))^{\mu}}{\Gamma(\mu+1)}E_{\mu}\left(\frac{K}{1-L}\left(\varphi(t)-\varphi(a)\right)^{\mu}\right),\quad t\in J.
$$
Therefore,
\begin{align}\label{62}
\left\Vert z-y\right\Vert _{C_{2-\xi ;\,\varphi }\left(J,\,\R\right)  }
&=\underset{t\in J 
}{\max }\left\vert \left( \varphi \left( t\right) -\varphi \left( a\right) \right)
^{2-\xi }(z(t)-y(t)) \right\vert\nonumber\\
&\leq \left( \varphi \left( t\right) -\varphi \left( a\right) \right)
^{2-\xi } \epsilon\,\,\frac{(\varphi(b)-\varphi(a))^{\mu}}{\Gamma(\mu+1)}E_{\mu}\left(\frac{K}{1-L}\left(\varphi(t)-\varphi(a)\right)^{\mu}\right)\nonumber\\
&= C_f \,\epsilon,
\end{align}
where $C_f:=\frac{(\varphi(b)-\varphi(a))^{\mu+2-\xi}}{\Gamma(\mu+1)}E_{\mu}\left(\frac{K}{1-L}\left(\varphi(b)-\varphi(a)\right)^{\mu}\right)$. Thus, the equation \eqref{eq1} is Ulam-Hyers stable.
\end{proof}

\begin{rem}
Define $\phi_f:\R_+\to\R_+$  by $\phi_f(\epsilon)=C_f\, \epsilon $. Then, $\phi_f\in C(\R_+,\,\R_+)$ and $\phi_f(0)=0$. Then \eqref{62} can be written as 
$$
\left\Vert z-y\right\Vert _{C_{2-\xi ;\,\varphi }\left(J,\,\R\right)  }\leq \phi_f(\epsilon). $$
Thus, the equation \eqref{eq1} is generalized Ulam-Hyers stable.
\end{rem}

\begin{theorem}[Ulam-Hyers-Rassias stability]\label{th6.2}
Assume that the  hypothesis $(H1)$  hold. Let there exists non decreasing function such $\chi\in C(J,\,\R_+)$  and $K^*>0$ such that
\begin{equation}\label{5.11}
\frac{1}{\Gamma(\mu)}\int_{a}^{t}\varphi'(s)(\varphi(t)-\varphi(s))^{\mu-1}\chi(s) ds\leq K^*\,\chi(t),\quad \text{for all}\,\, t\in J.
\end{equation} Then, the implicit FDE \eqref{eq1} is Ulam-Hyers-Rassias stable provided the condition \eqref{41} hold.
\end{theorem}
\begin{proof}
Let any $\epsilon>0$ and let $z \in C_{2-\xi ;\, \varphi }(J, \R)$ be any solution of  the inequality 
\begin{equation}\label{5.2a}
\left| ^H \mathcal{D}^{\mu,\,\nu\,;\, \varphi}_{a +}z(t)
- f(t, z(t),^H \mathcal{D}^{\mu,\,\nu\,;\, \varphi}_{a +}z(t))\right| \leq \epsilon\,\chi(t), \,\, t\in J.
\end{equation}
Then, proceeding as in the proof of Theorem \ref{th6.1}, we obtain
\begin{equation}\label{5.12}
\left|z(t)-(\varphi(t)-\varphi(a))^{\xi-1}\tilde{A}_z- \mathcal{I}_{a +}^{\mu\,;\, \varphi}g_z(t)\right| 
\leq \mathcal{I}_{a +}^{\mu\,;\, \varphi}\left| w(t)\right| \leq\epsilon \,\,\mathcal{I}_{a +}^{\mu\,;\, \varphi} \chi(t)
\leq \epsilon\, K^* \chi(t),\,\, t\in J.
\end{equation}
Taking $y\in C_{2-\xi}(J,\,\R)$ be any solution of the problem  \eqref{5.2a1}, and following similar steps as in the proof of  Theorem \ref{th6.1}, we obtain
$$
\left|z(t)-y(t)\right|\leq \epsilon\, K^* \chi(t)+\frac{K}{1-L}\frac{1}{\Gamma(\mu)}\int_{a}^{t}\varphi'(s)(\varphi(t)-\varphi(s))^{\mu-1}\left|z(s)-y(s) \right|ds,\,\, t\in J.
$$
By applying Gronwall inequality (Theorem \ref{lema4}), we get
$$
\left|z(t)-y(t)\right|
\leq\epsilon \,K^* \chi(t) E_{\mu}\left(\frac{K}{1-L}\left(\varphi(t)-\varphi(a)\right)^{\mu}\right),\,\, t\in J,
$$
which gives,
\begin{equation}\label{63}
\left|  \left(\varphi(t) -\varphi(t) \right) ^{2-\xi}\left( z(t)-y(t)\right) \right| 
\leq C_{f,\chi} \,\epsilon\,\chi(t),\,\, t\in J,
\end{equation}
where $C_{f,\chi}:=K^*(\varphi(b)-\varphi(a))^{2-\xi}E_{\mu}\left(\frac{K}{1-L}\left(\varphi(b)-\varphi(a)\right)^{\mu}\right)$. This proves, equation \eqref{eq1} is Ulam-Hyers-Rassias stable.
\end{proof}

\begin{rem}
The proof of the equation \eqref{eq1} is generalized Ulam-Hyers-Rassias stable, follows by taking $\epsilon=1$, in the inequality \eqref{63}.
\end{rem}

\section{Example}
\begin{ex}
Consider the following nonlocal BVP for implicit FDE involving $\varphi$-Hilfer fractional derivative 
\begin{align}
^H \mathcal{D}^{\mu,\,\nu\,;\, \varphi}_{0 +}y(t)
&=\frac{\cos t}{10 \,e^{t+1}} \left[ \sin y(t)+^H \mathcal{D}^{\mu,\,\nu\,;\, \varphi}_{0 +}y(t)\right] , ~t \in  (0,\,1],  ~1<\mu<2, ~0\leq\nu\leq 1,~\label{eq11}\\
y(0)&=0,\label{eq12}\\
y(1)&=\frac{10}{7}\,\mathcal{I}_{0 +}^{\frac{4}{5}\,;\, \varphi}y\left( \frac{1}{3}\right) +\frac{13}{6}\,\mathcal{I}_{0 +}^{\frac{8}{3}\,;\, \varphi}y\left( \frac{1}{2}\right) .\qquad  \label{eq13}
\end{align}
Define    $f: (0,\,1]\times \R^2  \ra \R $ by 
$$
f(t,\,u,\,v)=\frac{\cos t}{10 \,e^{t+1}} \left[ \sin u+v\right].
$$
 Then, for any $u_i, v_i\in \R(i=1,2)$ and $t\in (0,\,1]$ we have
\begin{align*}
\left| f(t,\,u_1,\,v_1)-f(t,\,u_2,\,v_2)\right| 
&\leq \frac{1}{10 \,e}\left\lbrace \left| \sin u_1-\sin u_2\right|+\left|  v_1-v_2 \right| \right\rbrace.
\end{align*}
Without loss of generality we may take $u_1<u_2$. 
Then, by Mean value theorem, $\exists \, \gamma\in(u_1,  u_2)$ such that,
$
 \sin u_1-\sin u_2= (u_1- u_2) \cos \gamma.
$
This gives,
$
\left| \sin u_1-\sin u_2 \right| \leq \left| u_1- u_2\right|.
$
Thus,
$$
\left| f(t,\,u_1,\,v_1)-f(t,\,u_2,\,v_2)\right| \leq  \frac{1}{10 \,e}\left\lbrace \left|  u_1- u_2\right|+\left|  v_1-v_2 \right| \right\rbrace.
$$
This proves $f$ satisfies the Lipschitz type condition in hypothesis $(H1)$ with  $K=L=\frac{1}{10 \,e}=0.0368.$ Further compering the problem \eqref{eq11}-\eqref{eq13}  to the problem \eqref{eq1}-\eqref{eq3} we have
 $a=0,\,b=1,\,\lambda_1=\frac{10}{7},\,\lambda_2=\frac{13}{6},\,\delta_1=\frac{4}{5},\,\delta_2=\frac{8}{3},\,\tau_1=\frac{1}{3},\,\tau_2=\frac{1}{2}$.  
 With these constants the equations \eqref{a1}, \eqref{41} and \eqref{42} becomes
\begin{align}\label{a11}
\sigma&=\frac{0.0368\,\Gamma(\xi-1)(\varphi(1)-\varphi(0))^{\mu}}{1-0.0368}\left\lbrace\frac{(\varphi(1)-\varphi(0))^{\xi-1}}{\Lambda\Gamma(\xi)}\Omega +\frac{1}{\Gamma(\xi+\mu-1)} \right\rbrace,
\end{align}
 where 
\begin{align}\label{a12}
\Lambda&=\frac{( \varphi (1)-\varphi (0))^{\xi -1}}{\Gamma(\xi)}-\left[ \frac{10}{7}\frac{(\varphi(\frac{1}{3})-\varphi(0))^{\xi -1+\frac{4}{5}}}{\Gamma(\xi+\frac{4}{5})}+\frac{13}{6}\frac{(\varphi(\frac{1}{2})-\varphi(0))^{\xi-1+\frac{8}{3}}}{\Gamma(\xi+\frac{8}{3})}\right]
\end{align}
and
\begin{align}\label{a13}
\Omega= \frac{10}{7}\frac{(\varphi(1)-\varphi(0))^{\frac{4}{5}}}{\Gamma(\xi-1+\mu+\frac{4}{5})}+\frac{13}{6}\frac{(\varphi(1)-\varphi(0))^{\frac{8}{3}}}{\Gamma(\xi-1+\mu+\frac{8}{3})}+\frac{1}{\Gamma(\xi+\mu-1)}.
\end{align}
Since all the assumptions of the Theorem \ref{th4.1} satisfied, the problem \eqref{eq11}-\eqref{eq13} has unique solution provided  $\sigma<1.$
Further, by Theorem \ref{th6.1}, for any solution $z \in C_{2-\xi ;\, \varphi }([0,1], \R)$ of  the inequality 
\begin{equation*}
\left| ^H \mathcal{D}^{\mu,\,\nu\,;\, \varphi}_{0 +}z(t)
- \frac{\cos t}{10 \,e^{t+1}} \left[ \sin z(t)+^H \mathcal{D}^{\mu,\,\nu\,;\, \varphi}_{0 +}z(t)\right] \right| \leq \epsilon,\,\, t\in [0,1],
\end{equation*}
there exists a unique solution $y \in C_{2-\xi ;\, \varphi }([0,1], \R)$ of the problem  \eqref{eq11}-\eqref{eq13} such that 
$$
\left\Vert z-y\right\Vert _{C_{2-\xi ;\,\varphi }\left(J,\,\R\right)  }
\leq C_f \,\epsilon,
$$
where $C_f=\frac{(\varphi(1)-\varphi(0))^{\mu+2-\xi}}{\Gamma(\mu+1)}E_{\mu}\left(\frac{0.0368}{1-0.0368}\left(\varphi(1)-\varphi(0)\right)^{\mu}\right)$.
\end{ex} 
Define 
\begin{equation}\label{43}
\chi(t)=E_{\mu}\left (\frac{1}{9}\left(\varphi(t)-\varphi(0)\right)^{\mu}\right),\, t\in [0,1].
\end{equation} Then,  $\chi:[0,\,1]\rightarrow \R$ is continuous non-decreasing function such that
$$
\mathcal{I}_{0 +}^{\mu\,;\, \varphi}E_{\mu}\left (\frac{1}{9}\left(\varphi(t)-\varphi(0)\right)^{\mu}\right)=\frac{1}{9}\left[ E_{\mu}\left (\frac{1}{9}\left(\varphi(t)-\varphi(0)\right)^{\mu}\right)-1\right] \leq \frac{1}{9}E_{\mu}\left (\frac{1}{9}\left(\varphi(t)-\varphi(0)\right)^{\mu}\right) .
$$
This proves  $\chi$ satisfies condition \eqref{5.11} with $K^*=\frac{1}{9}$. Therefore, by Theorem  \ref{th6.2}, for each solution  $z \in C_{2-\xi ;\, \varphi }([0,1], \R)$ of  the inequality 
\begin{equation*}\label{5.1}
\left| ^H \mathcal{D}^{\mu,\,\nu\,;\, \varphi}_{0 +}z(t)
- \frac{\cos t}{10 \,e^{t+1}} \left[ \sin z(t)+^H \mathcal{D}^{\mu,\,\nu\,;\, \varphi}_{0 +}z(t)\right] \right| \leq \epsilon E_{\mu}\left (\frac{1}{9}\left(\varphi(t)-\varphi(0)\right)^{\mu}\right),\,\, t\in [0,1],
\end{equation*}
there exists a unique solution $y \in C_{2-\xi ;\, \varphi }([0,1], \R)$ of the problem  \eqref{eq11}-\eqref{eq13} such that 
$$
\left|  \left(\varphi(t) -\varphi(a) \right) ^{2-\xi}\left( z(t)-y(t)\right) \right| 
\leq \epsilon C_{f,\,\chi} \, E_{\mu}\left (\frac{1}{9}\left(\varphi(t)-\varphi(0)\right)^{\mu}\right) ,
$$
where $C_{f,\,\chi} =\frac{1}{9}(\varphi(1)-\varphi(0))^{2-\xi}E_{\mu}\left(\frac{0.0368}{1-0.0368}\left(\varphi(1)-\varphi(0)\right)^{\mu}\right)$.

 \textbf{A particular case:} For $\mu=\frac{3}{2},\,\nu=1 $ and $\varphi(t)=t$, the problem \eqref{eq11}-\eqref{eq13}  reduces to the  following problem
\begin{align}
^C \mathcal{D}^{\frac{3}{2}}_{0 +}y(t)
&=\frac{\cos t}{10 \,e^{t+1}} \left[ \sin y(t)+^C \mathcal{D}^{\frac{3}{2}}_{0 +}y(t)\right] ,~\label{eq111}\\
y(0)&=0,\label{eq112}\\
y(1)&=\frac{10}{7}\,\mathcal{I}_{0 +}^{\frac{4}{5}}y\left( \frac{1}{3}\right) +\frac{13}{6}\,\mathcal{I}_{0 +}^{\frac{8}{3}}y\left( \frac{1}{2}\right), \qquad  \label{eq113}
\end{align}
which is the nonlocal BVP  for implicit FDEs involving Caputo fractional derivative.
In this case $\xi=\mu+\nu(2-\mu)=2$. Further putting the values of the constants $\xi, \mu$ with $\varphi(t)=t$ in \eqref{a11}, \eqref{a12} and \eqref{a13} has the following values
$
\Lambda=0.87045, \Omega=1.35464\,\,\text{and}\,\,\sigma=0.0881987<1.
$
It can be seen  that  all the assumptions of the Theorem  \ref{th4.1} and Theorem   \ref{th6.1}   are satisfied. Therefore by Theorem  \ref{th4.1}, the Caputo nonlocal implicit BVP \eqref{eq111}-\eqref{eq113} has at least one solution.  Further, for any solution $z\in C\left([0,\,1], \R \right) =\mathcal{C}$ of the inequality 
$$
\left| ^C \mathcal{D}^{\frac{3}{2}}_{0 +}z(t)
- \frac{\cos t}{10 \,e^{t+1}} \left[ \sin z(t)+^C \mathcal{D}^{\frac{3}{2}}_{0 +}z(t)\right] \right| \leq \epsilon, \,\, t\in [0,1],
$$ 
there is solution $y\in C\left([0,\,1], \R \right) =\mathcal{C}$  of problem \eqref{eq111}-\eqref{eq113} such that
$$
\left\Vert z-y\right\Vert _\mathcal{C}\leq C_f \,\epsilon,
$$
where $\left\Vert \cdot\right\Vert _\mathcal{C}$ is supremum norm on $\mathcal{C}$ and  the problem \eqref{eq111} is Ulam-Hyers stable. Further, for $\varphi(t)=t, \mu=\frac{3}{2}$ from equation \eqref{43} we have $\chi(t)= E_{\frac{3}{2}}\left (\frac{1}{9}t^{\frac{3}{2}}\right)$ which satisfy
$$
\mathcal{I}_{0 +}^{\frac{3}{2}} E_{\frac{3}{2}}\left (\frac{1}{9} t^{\frac{3}{2}}\right)
\leq \frac{1}{9} E_{\frac{3}{2}}\left (\frac{1}{9} t^{\frac{3}{2}}\right).
 $$
Therefore, by Theorem \ref{th6.2}, for each solution $z\in\mathcal{C}$ of the inequality 
$$
\left| ^C \mathcal{D}^{\frac{3}{2}}_{0 +}z(t)
- \frac{\cos t}{10 \,e^{t+1}} \left[ \sin z(t)+^C \mathcal{D}^{\frac{3}{2}}_{0 +}z(t)\right] \right| \leq \epsilon \chi(t),\,\, t\in [0,1],
$$ 
there is solution $y\in \mathcal{C}$  of problem \eqref{eq111}-\eqref{eq113} such that
$$
\left|  z(t)-y(t) \right| 
\leq \epsilon \,C_{f,\,\chi} \,\chi(t),
$$
and hence \eqref{eq111} is Ulam-Hyers-Rassias stable. 

\section*{Acknowledgment}
The second author  acknowledges the Science and Engineering Research Board (SERB), New Delhi, India for the Research Grant (Ref: File no. EEQ/2018/000407).

\end{document}